\documentclass[12pt,a4paper]{amsart}
\usepackage{graphicx,latexsym,amsfonts,amsmath,amssymb,rotating,txfonts,mathrsfs,enumerate}
\usepackage{epic}
\usepackage{curves}
\usepackage{pdfsync}
\input xy
\xyoption{all}

\newcommand{\Hom}{ \,{\rm Hom} \,}
\newcommand{\Sym}{ \,{\rm Sym} \,}

{\bf}{\rm}
\newtheorem{theorem}{Theorem}[section]
\newtheorem{proposition}[theorem]{Proposition}
\newtheorem{corollary}[theorem]{Corollary}
\newtheorem{lemma}[theorem]{Lemma}
\newtheorem{definition}[theorem]{Definition}
\newtheorem{remark}[theorem]{Remark}
\newtheorem{conjecture}[theorem]{Conjecture}

{\bf}{\it}

\newcommand{\RR}{{\mathbb R }}
\newcommand{\CC}{{\mathbb C }}

\newcommand{\ZZ}{{\mathbb Z }}
\newcommand{\PP}{ {\mathbb P }}

\newcommand{\GG}{{\mathbb G }}
\newcommand{\ff}{{\mathbf f }}

\newcommand{\calx}{\mathcal{X}}
\newcommand{\calo}{\mathcal{O}}

\newcommand{\tpo}{\mathrm{Tp}_O}
\newcommand{\reg}{\mathrm{reg}}

\newcommand{\gnk}{\mathrm{Diff}_n \times \mathrm{Diff}_m}

\newcommand{\Tp}{\mathrm{Tp}}
\newcommand{\bbb}{\mathbf{b}}
\newcommand{\mult}{\mathrm{mult}}
\newcommand{\epd}[1]{\mathrm{eP}[#1]}
\newcommand{\mdeg}[1]{\mathrm{mdeg}[#1]}
\newcommand{\emu}{\mathrm{emult}}
\newcommand{\um}[1]{\mathrm{sum}(#1)}
\newcommand{\dist}{{\mathrm{dst}}}
\newcommand{\lead}{\mathrm{lead}}
\newcommand{\Bipi}{{\boldsymbol{\Pi}_k}}
\newcommand{\bipi}{{\boldsymbol{\pi}}}
\newcommand{\sg}[1]{\mathcal{S}_{#1}}
\newcommand{\res}{\operatornamewithlimits{Res}}

\newcommand{\ires}{\res_{z_1=\infty}\res_{z_{2}=\infty}\dots\res_{z_k=\infty}}
\newcommand{\sires}{\res_{\mathbf{z}=\infty}}
\newcommand{\dbz}{\,d\mathbf{z}}

\newcommand{\coeff}{\mathrm{coeff}}

\newcommand{\symdot}{\mathrm{Sym}^{\le k}\CC^n}
\newcommand{\symdotx}{\mathrm{Sym}^{\le k}T_X^*}
\newcommand{\symdotxn}{\mathrm{Sym}^{\le n}T_X^*}

\newcommand{\grass}{\mathrm{Grass}}

\newcommand{\flag}{\mathrm{Flag}}
\newcommand{\TT}{\mathrm{T}}
\newcommand{\Fl}{\mathcal{F}}

\newcommand{\cotx}{T_X^*}

\newcommand{\bz}{\mathbf{z}}
\newcommand{\bi}{\mathbf{i}}
\newcommand{\baa}{\mathbf{a}}
\newcommand{\bj}{\mathbf{j}}

\newcommand{\bv}{\mathbf{v}}

\newcommand{\kt}{{K}}

\newcommand{\OO}{\mathcal{O}}

\newcommand{\jetk}[2]{J_{k}({#1},{#2})}

\newcommand{\jetreg}[2]{J_{k}^{\mathrm{reg}}({#1},{#2})}
\newcommand{\jetnondeg}[2]{J_{k}^{\mathrm{nondeg}}({#1},{#2})}
\newcommand{\nondeg}{\mathrm{nondeg}}

\newcommand{\tc}{\hat T}

\newcommand{\GL}{\mathrm{GL}}
\newcommand{\sym}{\mathrm{Sym}}

\def\a{\alpha}

\def\g{\gamma}
\def\d{\delta}

\def\l{\lambda}

\def\s{\sigma}

\def\vp{\varphi}

\title{Thom polynomials of Morin singularities and the Green-Griffiths-Lang conjecture} 

\author{Gergely B\'erczi}
\address{Mathematical Institute, University of Oxford, Andrew Wiles Building, OX2 6GG Oxford, UK\\ Tel.: +44-1865-273587 \\ \email{berczi@maths.ox.ac.uk},}
\thanks{This work was partially supported by the Engineering and Physical Sciences 
Research Council [grant numbers   GR/T016170/1,EP/G000174/1].}

\date{}
\begin{document}

\begin{abstract}
Green and Griffiths \cite{gg} and Lang \cite{lang} conjectured that for every complex projective algebraic variety $X$ of general type there exists a proper algebraic subvariety of $X$ containing all nonconstant entire holomorphic curves $f:\CC \to X$. We construct a compactification of the invariant jet differentials bundle over complex manifolds motivated by an algebraic model of Morin singularities and we develop an iterated residue formula using equivariant localisation for tautological integrals over it. We show that the polynomial GGL conjecture for a generic projective hypersurface of degree $\deg(X)>2n^{10}$ follows from a positivity conjecture for Thom polynomials of Morin singularities.  
\end{abstract}

\maketitle

\section{Introduction}\label{sec:intro}

The Green-Griffiths-Lang (GGL) conjecture \cite{gg,lang} states that every projective algebraic variety $X$ of general type contains a proper algebraic subvariety $Y\subsetneqq X$ such that every
nonconstant entire holomorphic curve $f:\CC \to X$ satisfies $f(\CC) \subset Y$. 
This conjecture is related to the stronger concept of a hyperbolic variety \cite{kob}. A projective variety $X$ is hyperbolic (in the sense of Brody) if there is no nonconstant entire holomorphic curve in $X$, i.e. any holomorphic map $f: \CC \to X$ must be constant.   Hyperbolic algebraic varieties have attracted considerable attention, in part because of their conjectured diophantine properties. For instance, Lang \cite{lang} has conjectured that any hyperbolic complex projective variety over a number field K can contain only finitely many rational points over K.
 
A positive answer to the GGL conjecture has been given for surfaces by McQuillan \cite{mcquillan} under the assumption that the second Segre number $c^2_1-c_2$ is positive. Siu in \cite{siu1,siu2,siu3,siu4} developed a strategy to establish algebraic degeneracy of entire holomorphic curves in generic hypersurfaces $X\subset \PP^{n+1}$ of high degree, and also hyperbolicity of such hypersurfaces for even higher degree. Following this strategy combined with techniques of Demailly \cite{dem} the first effective lower bound for the degree of the hypersurface in the GGL conjecture was given by Diverio, Merker and Rousseau in \cite{dmr} where they confirmed the conjecture for generic projective hypersurfaces $X\subset \PP^{n+1}$ of degree $\deg(X)>2^{n^5}$. In \cite{b1} the author introduces equivariant localisation on the Demailly-Semple tower and adapts the argument of \cite{dmr} to improve this lower bound to $\deg(X)>n^{8n}$. More recently  Demailly \cite{dem15} proved a directed version of the GGL conjecture for pairs $(X,V)$ satisfying certain jet stability conditions and announced the proof of the Kobayashi conjecture on the hyperbolicity of very general algebraic hypersurfaces and complete intersections. Siu \cite{siu4} proved the Kobayashi hyperbolicity of projective hypersurfaces of sufficiently high (but not effective) degree.   
 
In this paper we replace the Demailly-Semple bundle with a more efficient algebraic model coming from global singularity theory to explore deep connections of hyperbolicity questions and the theory of Thom polynomials of singularities. In particular, we prove the GGL conjecture for generic hypersurfaces with polynomial degree modulo a positivity conjecture. 

Proving algebraic degeneracy of holomorphic curves on $X$ means finding a nonzero polynomial function $P$ on $X$ such that all entire curves $f:\CC \to X$ satisfy $P(f(\CC))=0$. All known methods of proof at this date are based on establishing first the existence of certain algebraic differential equations $P(f,f',\ldots,f^{(k)})=0$ of some order $k$, and then finding enough such equations so that they cut out a proper algebraic locus $Y\subsetneqq X$. 

The central object in the study of polynomial differential equations is the bundle $J_kX$ of $k$-jets $(f',f'',\ldots, f^{(k)})$ of germs of holomorphic curves $f:\CC \to X$ over $X$ and the associated Green-Griffiths bundle  $E_{k,m}^{GG}=\calo(J_kX)$ of algebraic differential operators \cite{gg} whose elements are polynomial functions $Q(f',\ldots,f^{(k)})$ of weighted degree $m$.  In \cite{dem} Demailly introduced the subbundle $E_{k,m} \subset E_{k,m}^{GG}$ of jet differentials that are invariant under reparametrization of the source $\CC$. The group $\GG_k$ of $k$-jets of reparametrisation germs $(\CC,0) \to (\CC,0)$ at the origin acts fibrewise on $J_kX$ and $\oplus_{m=1}^\infty E_{k,m}=\calo(J_kX)^{U_k}$ is the graded algebra of invariant jet differentials under the maximal unipotent subgroup $U_k$ of $\GG_k$. This bundle gives a better reflection of the geometry of entire curves, since
it only takes care of the image of such curves and not of the way they are
parametrized. However, it also comes with a technical difficulty, namely, the reparametrisation group $\GG_k$ is non-reductive, and the classical geometric invariant theory of Mumford \cite{git} is not applicable to describe the invariants and the quotient $J_kX/\GG_k$; for details see \cite{bk,dk}. In \cite{dem} Demailly describes a  smooth compactification of $J_kX/\GG_k$ as a tower of projectivised bundles on $X$---the Demailly-Semple bundle---endowed with tautological line bundles $\tau_1,\ldots \tau_k$ whose sections are $\GG_k$-invariants. Global sections of properly chosen twisted tautological line bundles over the Demailly-Semple bundle give algebraic differential equations of degree $k$ which all $k$-jets of holomorphic curves must satisfy. An algebraic version of Demailly's holomorphic Morse inequalities in \cite{dmr} reduce the existence of global sections to the positivity of a certain intersection number on the Demailly-Semple tower. 

The second ingredient of the strategy---which was established by Siu \cite{siu1} based on earlier works of Voisin \cite{voisin}, and then turned into a final form in \cite{dmr} ---is the deformation of the found global sections of the invariant jet differentials bundle by means of slanted vector fields having low pole order to produce, by plain differentiation,
many new algebraically independent invariant jet differentials, which then force entire curves to lie in a proper closed subvariety $Y\subsetneq X$. The degree $\deg(X)$ of those hypersurfaces where the GGL conjecture follows depends on $k$, the degree of the differential equation we start with, and therefore it is crucial to keep this degree low and to find global sections of degree comparable with the dimension $n$ of $X$. 

Demailly in \cite{dem3} used probabilistic methods to find differential equation of order $k\gg n$ for a projective directed manifold $(X,V)$ with $K_V$ big.  Merker \cite{merker3} proved the existence of global differential equations of high degree for projective hypersurfaces in $\PP^{n+1}$ of degree at least $n+3$ using algebraic Morse inequalities. In \cite{b1} the author introduces equivariant localisation on the Demailly-Semple tower and develops residue formulas to reduce the complexity of cohomological computations of \cite{dmr}.  Darondeau \cite{darondeau} further improved techniques of \cite{b1} to study algebraic degeneracy of entire curves in complements of smooth projective hypersurfaces.

In this paper we substitute the Demailly-Semple bundle with a new fibrewise compactification $\pi: \tilde{\calx_k} \to X$ of $J_kX/\GG_k$ endowed with a tautological line bundle $\tau$. The construction is motivated by the author's earlier work in global singularity theory \cite{bsz} on Thom polynomials of singularity classes. This algebraic model gives a better reflection of the geometry of the jet differentials bundle and it establishes a strong link between hyperbolic varieties and topological invariants of singularitites. The main technical result of the present paper is  
the following iterated residue formula for tautological integrals on $\tilde{\calx_k}$ for any smooth projective variety $X$ (not just hypersurfaces) which follows from a generalised and improved version of the main technical vanishing theorem in \cite{bsz}.  

\begin{theorem}\label{maintechnical}
Let $X$ be a smooth projective variety and let  $u=c_1(\tau)$ and $h=c_1(\pi^* \calo_{X}(1))$ denote the first Chern classes of the tautological line bundle on $\tilde{\calx}_k$  and the (pull-back of the) hyperplane line bundle on $X$, respectively. For any homogeneous polynomial $P=P(u,h)$ of degree $\deg(P)=\dim \tilde{\calx}_k=n+k(n-1)$ we have
\begin{equation} \label{intnumbertwo}
\int_{\tilde{\calx}_k}P(u,h)=\int_X\sires \frac{
Q_k(z_1,\ldots, z_k)\,\prod_{m<l}(z_m-z_l) P(z_1+\ldots +z_k,h)}{
\prod_{m+r \le l \le k}
(z_m+z_r-z_l)  (z_1\ldots z_k)^n} \prod_{j=1}^k s\left(\frac{1}{z_j}\right)
\end{equation}
where $s\left(\frac{1}{z_j}\right)=1+\frac{s_1(X)}{z_j}+\frac{s_2(X)}{(z_j^2}+\ldots +\frac{s_n(X)}{(z_j^n}$ is the total Segre class of $X$ at $1/z_j$, the iterated residue is equal to the coefficient of $(z_1\ldots z_k)^{-1}$ in the Laurent expansion of the rational expression in the domain $z_1\ll \ldots \ll z_k$ and finally $Q_k$ is a polynomial invariant of Morin singularities. 
\end{theorem}

Let us explain briefly the nature and origin of the polynomial invariants $Q_k$ in this formula and its link to Thom polynomials. Let
 $f:N\to M$ be a holomorphic map between two complex manifolds,
of dimensions $n\leq m$. We say that $p\in N$ is a {\em singular}
point of $f$ (or $f$ has a singularity at $p$) if the rank of the
differential $df_p:\TT_pN\to \TT_{f(p)}M$ is less than $n$. The topology
of the situation often forces $f$ to be singular at some points of
$N$. To introduce a finer classification of singular points, choose local
coordinates near $p\in N$ and $f(p)\in M$, and consider the resulting
map-germ $f_p:(\CC^n,0)\to(\CC^m,0)$, which may be thought of as a
sequence of $m$ power series in $n$ variables without constant
terms. Let $\mathrm{Diff}_n$ denote the group of germs of local holomorphic reparametrisations $(\CC^n,0) \to (\CC^n,0)$. Then $\gnk$ acts
on the space $J(n,m)$ of all such map-germs. 
We call $\gnk$-invariant subsets
$O\subset J(n,m)$ {\em singularities}.  For a singularity $O$ and
holomorphic $f:N\to M$, we can define the set
\[ Z_O[f] = \{p\in N;\; f_p\in O\}, \] 
which is independent of any
coordinate choices.  Then, under some additional technical
assumptions ($N$ compact, appropriately chosen closed $O$, and $f$
sufficiently generic), $Z_O[f]$ is an analytic subvariety of $N$.  The
computation of the Poincar\'e dual class $\alpha_O[f]\in H^*(N,\ZZ)$ of
this subvariety is one of the fundamental problems of global singularity
theory. The following result is called Thom's principle in the literature: 
\\ {\em For appropriate $\gnk$-invariant $O$ of codimension $j$ in $J(n,m)$, there exists a
  homogeneous polynomial $\tpo\in \CC[t_1,\ldots, t_j]$ of degree $j$---the Thom polynomial of $O$--- such that for an arbitrary, sufficiently generic map $f:N\to M$, the cycle
  $Z_f[O]\subset N$ is Poincar\'e dual to the characteristic class
  $\tpo(c_1(TN-f^*TM),\ldots ,c_j(TN-f^*TM))$. }
  
The computation of these polynomials is a central problem in singularity theory, see \cite{thom,kazarian2,FR1,rf,rimanyi}.
For a map germ $f\in J(n,m)$ we can associate the finite dimensional nilpotent algebra $A_f$ defined as the quotient of the algebra of power series with no constant term $\CC_0[[x_1, . . . , , x_n]]$ by the ideal generated by the pull-back subalgebra $f^*(\CC_0[[y_1,...,,y_m]])$. Then the classes
\[O_k=\{f\in J(n,m):A_f \simeq t\CC[t]/t^{k+1}\}\]
are called Morin singularities.

The link of Morin singularities to the GGL conjecture becomes clear from an algebraic characterization of $O_k$ due to Gaffney \cite{gaffney,bsz}. This 'test curve model' says that an element $f$ of (an open dense subset of) $J(n,m)$ lies in $O_k$ if and only if there exist a test curve $\g \in J(1,n)$ such that the $k$-jet of $f \circ \g$ is $0$. Reparametrisation of the test curve by an element of the group $\GG_k$ of $k$-jets of reparametisations $(\CC,0) \to (\CC,0)$ is again a test curve, and therefore a dense open subset of $O_k$ fibres over the quotient $\jetreg 1n/\GG_k$, where $\jetreg 1n$ is the set of $k$-jets of germs in $J(1,n)$ with non-vanishing linear part. Note, however,  that the fibres of the quotient $J_kX/\GG_k$ are isomorphic to $\jetreg 1n/\GG_k$.

In \cite{bsz} we described a projective completion of the quotient $\jetreg 1n/\GG_k$, and applied equivariant localisation on this compactification to get the Thom polynomial of Morin singularities in the following iterated residue form: 
\begin{equation}\label{thompolynomial}
\Tp_k^{m-n}(c_1,\ldots)=\sires \frac{(-1)^k\prod_{m<l\le k}(z_m-z_l)\,Q_k(z_1\ldots  z_k)}
{\prod_{m+r \le l\le k}(z_m+z_r-z_l)}
\prod_{l=1}^k c\left(\frac1{z_l}\right)\,z_l^{m-n}\;dz_l,
\end{equation}
where $c\left(\frac1{z_l}\right)=1+\frac{c_1}{z_l}+\frac{c_2}{z_l^2}+\ldots$
is the total Chern class of $TN-f^*TM$, and $Q_k$ is a homogeneous polynomial defined as the dual of a certain Borel orbit, see Remark \ref{remarkq} for the definition. 
The coefficients are therefore encoded in the Thom generating function
\begin{equation}\label{tpgenerating}
\Tp_k(z_1,\ldots, z_k)=\frac{\prod_{m<l\le k}(z_m-z_l)\,Q_k(z_1\ldots  z_k)}
{\prod_{m+r \le l\le k}(z_m+z_r-z_l)},
\end{equation}
whose numerator and denominator are homogeneous polynomials of equal degree and therefore its expansion in the domain $1\ll|z_1|\ll \ldots \ll |z_k|$ gives terms of the form $\bz^{\bi}=z_1^{i_1} \cdots z_k^{i_k}$ satisfying $\Sigma \bi=i_1+\ldots +i_k=0$ multiplied by some integer coefficient $\Tp_{\bi}$. These coefficients are topological invariants of Morin singularities and they have attracted considerable attraction. 

Any integer vector $\bi \in \ZZ^k$ can be uniquely written as the difference $\bi=\bi^+-\bi^-$ of nonnegative vectors $\bi^+,\bi^- \in \ZZ_{\ge 0}^k$. For a nonzero vector $\bi$ with $\Sigma \bi=0$ we call $\bj=\bj^+-\bj^-$ a {\it predecessor} of $\bi$ if $\Sigma \bj=0$ and $\bi^+=\bj^++e_s$ for some $s$ such that $i_s=\max_{1\le t \le k} i_t$ is a largest (positive) coordinate of $\bi$. Here $e_s=(0,\ldots,1,\ldots, 0)$ is the $s$th basis vector with all but the $s$th coordinate $0$.  The second part of the following conjecture suggests that there are no isolated nonzero coefficients of $\Tp_k$ and it is further motivated in Sect.~\ref{sec:conjecture}.  
\begin{conjecture}[Generalised positivity conjecture]\label{conj}  Let $\Tp_\bi$ denote the coefficient of $\bz^{\bi}$ of $\Tp_k(\bz)$ expanded in the domain $|z_1|\ll\dots\ll|z_d|$. Then  
\begin{enumerate}
\item (Rim\'anyi \cite{rimanyi}) $\Tp_\bi\ge 0$ for any $\bi$.
\item (Connectedness of positive coefficients) If $\Tp_\bi>0$ then $\bi$ has a predecessor $\bj$ such that  $\Tp_{\bj}>0$ and $\frac{\Tp_{\bi}}{\Tp_{\bj}}<k^2$.
\end{enumerate}
\end{conjecture}
Global sections of properly chosen twisted tautological line bundles over $\tilde{\calx}_k$ give algebraic differential equations of degree $k$. Following \cite{dmr} and using Morse inequalities we deduce the existence of these global sections from the positivity of a well-defined tautological integal over $\tilde{\calx}_k$ We apply Theorem \ref{maintechnical} to prove the positivity of this integral at the critical degree $k=n$. 
Following \cite{dmr} this implies 
\begin{theorem}[Partial answer to polynomial GGL conjecture]\label{mainthmtwo}
Let $X\subset \PP^{n+1}$ be a generic smooth projective hypersurface
of degree $\deg(X)\ge 2n^{10}$. Then Conjecture \ref{conj} implies the existence of a proper algebraic subvariety $Y\subsetneqq X$ such
that every nonconstant entire holomorphic curve $f:\CC \to X$ has
image contained in $Y$.
\end{theorem}  
This theorem confirms the Green-Griffiths-Lang conjecture for generic hypersurfaces with polynomial degree modulo a positivity conjecture in global singularity theory, which is an interesting link between at first sight unrelated fields of mathematics. 

\noindent\textbf{Acknowledgments} I would like to thank Damiano Testa and Frances Kirwan for patiently listening to details of this work. The first version of this paper was presented in Strasbourg, Orsay and Luminy in 2010/2011. I would like to thank to Jean-Pierre Demailly, Jo\"el Merker, Simone Diverio, Erwan Rousseau and Lionel Darondeau for their comments and suggestions. The paper has been rewritten based on these discussions to make the technical details of localisation more available to non-experts. 
The author warmly thanks Andr\'as Szenes, his former PhD supervisor, for the collaboration on \cite{bsz}, from which this paper has outgrown.

\section{Jet differentials}\label{sec:jetdiff}

The central object of this paper is the algebra of invariant jet differentials under
reparametrisation of the source space $\CC$. For more details see the survey papers \cite{dem,dr}.

\subsection{Invariant jet differentials}\label{subsec:jetdiff}
Let $X$ be a complex $n$-dimensional manifold and let $k$ be a positive integer. Green and Griffiths
in \cite{gg} introduced the bundle $J_kX \to X$
of $k$-jets of germs of parametrized curves in $X$; its
fibre over $x\in X$ is the set of equivalence classes of germs of holomorphic
maps $f:(\CC,0) \to (X,x)$, with the equivalence relation $f\sim g$
if and only if the derivatives $f^{(j)}(0)=g^{(j)}(0)$ are equal for
$0\le j \le k$. If we choose local holomorphic coordinates
$(z_1,\ldots, z_n)$ on an open neighbourhood $\Omega \subset X$
around $x$, the elements of the fibre $J_kX_x$ are represented by the Taylor expansions 
\[f(t)=x+tf'(0)+\frac{t^2}{2!}f''(0)+\ldots +\frac{t^k}{k!}f^{(k)}(0)+O(t^{k+1}) \]
 up to order $k$ at $t=0$ of $\CC^n$-valued maps $f=(f_1,f_2,\ldots, f_n)$
on open neighbourhoods of 0 in $\CC$. Locally in these coordinates the fibre  can be written as
\[J_kX_x=\left\{(f'(0),\ldots, f^{(k)}(0)/k!)\right\}=(\CC^n)^k,\]
which we identify with $\CC^{nk}$.  Note that $J_kX$ is not
a vector bundle over $X$ since the transition functions are polynomial but
not linear, see \cite{dem} for details.

Let $\GG_k$ denote the group of $k$-jets of local reparametrisations
of $(\CC,0) \to (\CC,0)$
\[t \mapsto \varphi(t)=\a_1t+\a_2t^2+\ldots +\a_kt^k,\ \ \ \a_1\in
\CC^*,\a_2,\ldots,\a_k \in \CC,\] 
under composition modulo terms $t^j$ for $j>k$. This group acts fibrewise on
$J_kX$ by substitution. A short computation shows that this is a
linear action on the fibre:
\begin{multline}\nonumber f \circ
\varphi(t)=f'(0)\cdot(\a_1t+\a_2t^2+\ldots
+\a_kt^k)+\frac{f''(0)}{2!}\cdot (\a_1t+\a_2t^2+\ldots
+\a_kt^k)^2+\ldots \\
\ldots +\frac{f^{(k)}(0)}{k!}\cdot (\a_1t+\a_2t^2+\ldots +\a_kt^k)^k 
\text{ modulo } t^{k+1}
\end{multline}
so the linear action of $\varphi$ on the $k$-jet $(f'(0),f''(0)/2!,\ldots,
f^{(k)}(0)/k!)$ is given by the following matrix multiplication:
\begin{equation*}
(f'(0),f''(0)/2!,\ldots,f^{(k)}(0)/k!) \cdot 
\left(\begin{array}{ccccc}
\a_1 & \a_2 & \a_3 & \cdots  & \a_k \\
0        & \a_1^2 & 2\a_1\a_2 & \cdots &  \a_1\a_{k-1}+\ldots +\a_{k-1}\a_1 \\
0        & 0       & \a_1^3  & \cdots & 3\a_1^2\a_{k-2}+\ldots \\
\cdot    & \cdot   & \cdot    & \cdot &  \cdot \\
0 & 0 & 0 & \cdots  & \a_1^k 
\end{array} \right)
\end{equation*}
where the matrix has general entry 
\[(\GG_{k})_{i,j}=\sum_{\substack{s_1,\ldots s_i \in \ZZ_+ \\
s_1+\ldots +s_i=j}}\a_{s_1}\ldots \a_{s_i} \text{ for } 1\le i,j\le k.\] 
$\GG_k$ sits in an exact sequence of groups
$1 \rightarrow U_k \rightarrow \GG_k \rightarrow \CC^* \rightarrow
1$,
 where $\GG_k \to \CC^*$ is the morphism $\varphi \to
\varphi'(0)=\a_1$ in the notation used above, and
\begin{equation}\label{uuk}
\GG_k=U_k \rtimes \CC^*
\end{equation}
is a $\CC^*$-extension of the unipotent group $U_k$. With the above identification, $\CC^*$ is
the subgroup of diagonal matrices satisfying $\a_2=\ldots =\a_k=0$ and
$U_k$ is the unipotent radical of $\GG_k$, consisting of matrices of the form above with $\a_1=1$. The action
of  $\l \in \CC^*$ on $k$-jets is thus described by
\[\l\cdot (f',f'',\ldots ,f^{(k)})=(\l f',\l^2 f'',\ldots,
\l ^kf^{(k)})\]

Following \cite{dem}, we introduce the Green-Griffiths vector bundle $E_{k,m}^{GG}$ whose fibres are complex-valued polynomials $Q(f',f'',\ldots ,f^{(k)})$ on the fibres of $J_kX$ of weighted degree $m$ with respect to the $\CC^*$ action above, that is, they satisfy
\[Q(\l f',\l^2 f'',\ldots, \l^k f^{(k)})=\l^m Q(f',f'',\ldots,
f^{(k)}).\]
The fibrewise $\GG_k$ action on $J_kX$ induces an action on $E_{k,m}^{GG}$. Demailly in \cite{dem} defined the  bundle of invariant jet differentials of order $k$ and weighted degree $m$ as the subbundle $E_{k,m}^n\subset E_{k,m}^{GG}$ of polynomial differential operators $Q(f,f',\ldots, f^{(k)})$ which are invariant under $U_k$, that is for any $\varphi \in  \GG_k$ 
\[Q((f\circ \varphi)',(f\circ \varphi)'', \ldots, (f\circ
\varphi)^{(k)})=\varphi'(0)^m\cdot Q(f',f'',\ldots, f^{(k)}).\]
We call $E_k^n=\oplus_m E_{k,m}^n=(\oplus_mE_{k,m}^{GG})^{U_k}$ the Demailly-Semple bundle of invariant jet differentials

\section{Compactification of the jet differentials bundle}\label{quotient}\label{sec:snowman}

In this section we construct a new fibred projective completion of $J_kX/\GG_k$ motivated by an algebraic model for Morin singularities in global singularity theory, the so called 'test curve model' of Gafffney \cite{gaffney}. Since $\GG_k$ acts on $J_kX$ fibrewise, we construct the quotient 
$J_kX_x/\GG_k$ of the fibre of $J_kX$ by $\GG_k$ first.  

If $u,v$ are positive integers let $J_k(u,v)$ denote the vector
space of $k$-jets of holomorphic maps $(\CC^u,0) \to (\CC^v,0)$ at
the origin, that is, the set of equivalence classes of maps
$f:(\CC^u,0) \to (\CC^v,0)$, where $f\sim g$ if and only if
$f^{(j)}(0)=g^{(j)}(0)$ for all $j=1,\ldots ,k$.
With this notation, the fibres of $J_kX$ are isomorphic to
$J_k(1,n)$, and the group $\GG_k$ is simply $\jetk 11$ with the
composition action on itself.

If we fix local coordinates $z_1,\ldots, z_u$ at $0\in \CC^u$ we can
again identify the $k$-jet of $f$ with the set of derivatives at the
origin, that is $(f'(0),f''(0),\ldots, f^{(k)}(0))$, where
$f^{(j)}(0)\in \mathrm{Hom}(\mathrm{Sym}^j\CC^u,\CC^v)$. This way we
get the equality
\[J_k(u,v)=\oplus_{j=1}^k\mathrm{Hom}(\mathrm{Sym}^j\CC^u,\CC^v)\]
One can compose map-jets via substitution and elimination of terms
of degree greater than $k$; this leads to the composition maps
\begin{equation}
  \label{comp}
\jetk vw \times \jetk uv \to \jetk uw,\;\;  (\Psi_2,\Psi_1)\mapsto
\Psi_2\circ\Psi_1 \mbox{modulo terms of degree $>k$ }.
\end{equation}
When $k=1$, $J_1(u,v)$ may be identified with $u$-by-$v$ matrices,
and \eqref{comp} reduces to multiplication of matrices.

The $k$-jet of a curve $(\CC,0) \to (\CC^n,0)$ is simply an element
of $J_k(1,n)$. We call such a curve $\gamma$ {\em regular}, if
$\gamma'(0)\neq 0$; introduce the notation $\jetreg 1n$ for the set
of regular curves:
\[\jetreg 1n=\left\{\g \in \jetk 1n; \g'(0)\neq 0 \right\}\]
Let $N \ge n$ be any integer and define
\[\Theta_k=\left\{\Psi\in J_k(n,N):\exists \g \in \jetreg 1n: \Psi \circ \g=0
\right\}\]
In words: $\Theta_k$ is the set of those $k$-jets of maps, which take at
least one regular curve to zero. By definition, $\Theta_k$ is the image
of the closed subvariety of $\jetk nN \times \jetreg 1n$ defined by
the algebraic equations $\Psi \circ \g=0$, under the projection to
the first factor. If $\Psi \circ \gamma=0$, we call $\g$ a {\em test
curve} of $\Theta$. This term originally comes from global singularity
theory as explained below.

A basic but crucial observation is the following. If $\g$ is a test
curve of $\Psi \in \Theta_k$, and $\vp \in \jetreg 11=\GG_k$ is a 
holomorphic reparametrisation of $\CC$, then $\g \circ \vp$ is,
again, a test curve of $\Psi$:
\begin{displaymath}
\label{basicidea}
\xymatrix{
  \CC \ar[r]^\vp & \CC \ar[r]^\g & \CC^n \ar[r]^{\Psi} & \CC^N}
\end{displaymath}
\[\Psi \circ \g=0\ \ \Rightarrow \ \ \ \Psi \circ (\g \circ \vp)=0\]

In fact, we get all test curves of $\Psi$ in this way if the
following open dense property holds: the linear part of $\Psi$ has
$1$-dimensional kernel. Before stating this in Theorem 
\ref{embedfinal} below, let us write down the equation $\Psi \circ
\g=0$ in coordinates in an illustrative case. Let
$\g=(\g',\g'',\ldots, \g^{(k)})\in \jetreg 1n$ and
$\Psi=(\Psi',\Psi'',\ldots, \Psi^{(k)})\in \jetk nN$ be the
$k$-jets of the test curve $\g$ and the map $\Psi$ respectively. Using the chain rule and the notation $v_i=\g^{(i)}/i!$, the equation $\Psi \circ \g=0$ reads
as follows for $k=4$:
\begin{eqnarray} \label{eqn4}
& \Psi'(v_1)=0, \\ \nonumber & \Psi'(v_2)+\Psi''(v_1,v_1)=0, \\
\nonumber
& \Psi'(v_3)+2\Psi''(v_1,v_2)+\Psi'''(v_1,v_1,v_1)=0, \\
&
\Psi'(v_4)+2\Psi''(v_1,v_3)+\Psi''(v_2,v_2)+
3\Psi'''(v_1,v_1,v_2)+\Psi''''(v_1,v_1,v_1,v_1)=0.
\nonumber
\end{eqnarray}


\begin{lemma}[\cite{gaffney,bsz}]\label{explgp} Let
$\g=(\g',\g'',\ldots, \g^{(k)})\in \jetreg 1n$ and
$\Psi=(\Psi',\Psi'',\ldots, \Psi^{(k)})\in \jetk nN$ be $k$-jets.
Then substituting $v_i=\g^{(i)}/i!$, the equation $\Psi\circ \g$ is equivalent to
  the following system of $k$ linear equations with values in
  $\CC^N$:
\begin{equation}
  \label{modeleq}
\sum_{\Sigma\tau=m} \Psi(\bv_\tau)=0,\quad m=1,2,\dots, k.
\end{equation}
\end{lemma}
For a given $\g \in \jetreg 1n$ and $1\le i \le k$ let $\mathcal{S}^i_{\g}$ denote the set of
solutions of the first $i$ equations in \eqref{modeleq}, that is,
\[\mathcal{S}^i_\g=\left\{\Psi \in \jetk nN;\Psi \circ \g=0 \text{ up to order } i \right\}\]
The equations \eqref{modeleq} are linear in $\Psi$, hence
\[\mathcal{S}^i_\g \subset \jetk nN\]
is a linear subspace of codimension $iN$, i.e a point of $\grass_{\mathrm{codim}=iN}(J_k(n,N))$, whose dual, $(\mathcal{S}^i_{\g})^*$, is an $iN$-dimensional subspace of $J_k(n,N)^*$. These subspaces are invariant under the reparametrization of $\g$. Note that for $N=1$ $J_k(n,1)$ can be identified with $\Hom(\CC^k,\symdot)$ where $\symdot=\bigoplus_{i=1}^k \sym^i \CC^n$.  Furthermore, for $\g \in J_k(1,n)$ we put $\g^{(i)}/i! \in \CC^n$ in the $i$th column of a matrix, then  $\jetreg 1n$ is identified with elements of $\Hom(\CC^k,\CC^n)$ with nonzero first column.   

\begin{theorem}[\cite{bsz}]\label{embedfinal} The map 
\[\phi: \jetreg 1n \rightarrow \flag_k(\symdot)
\]
\[\gamma  \mapsto \Fl_\g=(\mathcal{S}^1_{\gamma})^*\subset \ldots \subset (\mathcal{S}^k_{\gamma})^*)\]
is $\GG_k$-invariant and induces an injective map on the $\GG_k$-orbits into the flag manifold 
\[\phi: \jetreg 1n /\GG_k \hookrightarrow \flag_k(\symdot).\]
Moreover, all these maps are $\GL(n)$-equivariant with respect to the standard action of $\GL(n)$ on $\jetreg 1n \subset \Hom(\CC^k,\CC^n)$ and the induced action on $\grass_k(\symdot)$.
\end{theorem}
For a point $\gamma\in J_k^\reg(1,n)$ let $v_i=\frac{\g^{(i)}}{i!}\in \CC^n$ denote the normed $i$th derivative. Then for $1\le i \le k$ (see \cite{bsz}):
\begin{equation}\label{sgamma}
\mathcal{S}^i_\g=\mathrm{Span}_\CC (v_1,v_2+v_1^2,\ldots, \sum_{j_1+\ldots +j_s=i}v_1^{j_1}\cdots v_s^{j_s})\subset \symdot.
\end{equation} 
Since $\phi$ is $\GL(n)$-equivariant, for $k\le n$ the image $\phi(\jetreg 1n/\GG_k) \subset \flag_k(\symdot)$ is the $\GL(n)$-orbit of $p_k=\phi(e_1,\ldots, e_k)$, that is 
\[\phi(\jetreg 1n)=\mathrm{GL_n} \cdot p_k\] 
with a highly singular closure 
\[X_{k}=\overline{\mathrm{GL_n} \cdot p_k} \subset \flag_k(\symdot). \] 
Let $P_{n,k} \subset \mathrm{GL}_n$ denote the parabolic subgroup which preserves the flag 
\[\mathbf{f}=(\mathrm{Span}(e_1)   \subset \mathrm{Span}(e_1,e_2) \subset \ldots \subset \mathrm{Span}(e_1,\ldots, e_k) \subset \CC^n).\] 
Define the partial desingularization 
\[\tilde{X}_k=\mathrm{GL}_n \times_{P_{n,k}} \overline{P_{n,k} \cdot p_k}\]
with the resolution map $\rho: \tilde{X}_k \to X_k$ given by $\rho(g,x)=g\cdot x$. 

Equivalently, let $\jetnondeg 1n \subset \jetreg 1n$ be the set of test curves with $\g',\ldots, \g^{(k)}$ linearly independent. These correspond to the regular $n \times k$ matrices in $\Hom(\CC^k,\CC^n)$, and they fibre over the set of complete flags in $\CC^n$:
\[\jetnondeg 1n \to \Hom(\CC^k,\CC^n)/B_k=\flag_k(\CC^n)\]
where $B_k \subset \GL(k)$ is the upper Borel. The image of the fibres under $\phi$ are isomorphic to $P_{n,k} \cdot p_k$, and therefore $\tilde{X}_k$ is the fibrewise compactification of $\jetnondeg 1n$ over $\flag_k(\CC^n)$.

In \cite{bsz} the authors develop an iterated residue formula based on equivariant localisation on $\tilde{X}_k$ to compute multidegrees of Morin singularity classes. We will generalise this method in the next section to compute cohomological pairings on $\tilde{X}_k$. 
\begin{remark} 
Note that the map $\gamma \mapsto (\mathcal{S}_\g^k)^*$ defines an embedding of the orbit set
\[\phi^{\grass}:\jetreg 1n /\GG_k \hookrightarrow \grass_k(\symdot)\]
into the Grassmannian of $k$-spaces in $\symdot$, which composed with the Veronese embedding then identifies $\overline{\jetreg 1n /\GG_k}$ as a subvariety of $\PP(\wedge^k \symdot)$.  
As explained in \cite{bk,b}, this variety is isomorphic to the curvilinear component of the punctual Hilbert scheme of $k$ points on $\CC^n$. 
\end{remark}
Let now $X\subset \PP^{n+1}$ be a smooth projective hypersurface of degree
$d$. We introduce the notation
\[\symdotx=T_X^*\oplus \sym^2(T_X^*)\oplus \ldots \oplus \sym^k(T_X^*).\]
Theorem \ref{embedfinal} gives us the following fibrewise embedding 
\begin{proposition}
The quotient $J_k(\cotx)/\GG_k$ has the structure of a locally trivial
bundle over $X$, and Theorem \ref{embedfinal} gives us a holomorphic embedding
\[\phi:J_k(\cotx)/\GG_k \hookrightarrow \flag_k(\symdotx)\] into the flag bundle of $\symdotx$ over $X$. The
fibrewise closure $\calx_k=\overline{\phi(J_k(\cotx))}$ of the image is a relative compactification of $J_k(\cotx)/\GG_k$ over $X$.
\end{proposition}
We can define a fibred version of $\tilde{X}$ too, a fibrewise partial desingularization 
\begin{equation}\label{desing}
\rho:\tilde{\calx_k} \to \calx_k
\end{equation}
over $X$, where $\tilde{\calx_k}$ is a locally trivial bundle over $X$ with fibres isomorphic to $\tilde{X}_k$. Indeed let $J_k^{\nondeg}(T_X^*)$ be the subbundle whose fibre over $x\in X$ is $J_k^\nondeg(\CC,T_{X,x})$. It fibres over the flag bundle 
\[J_k^\nondeg(T_X^*) \to \flag_k(T_X^*),\]
and the fibrewise compactification in $\flag_k(\symdotx)$ gives us $\tilde{\calx}_k$.
In the next section, following \cite{bsz}, we develop an equivariant localisation formula on $\tilde{\calx}_{k}$ to compute topological intersection numbers, leading us to an iterated residue formula.			

\section{Equivariant localisation on $\tilde{\calx}_{k}$}\label{sec:locsnowman}

Let $\calo_{\calx_{k}}(1)=\calo_{\PP(\wedge^k(\symdotx))}(1)|_{\calx_k}$ be the determinant bundle of the canonical rank $k$ bundle over $\flag_k(\symdotx)$ restricted to $\calx_k$.  Let $\calo_{\tilde{\calx}_{k}}(1)=\rho^*\calo_{\calx_{k}}(1)$ be its pull-back to $\tilde{\calx}_k$, moreover $\pi:\calx_k \to X$ and $\tilde{\pi}=\pi \circ \rho:\tilde{\calx}_k \to X$ be the projections onto $X$. Let
\begin{equation}\label{uhnotations}
u=c_1(\calo_{\tilde{\calx}_k}(1)), h=c_1(\calo_X(1))
\end{equation}
denote the first Chern classes of the corresponding line bundles. In this section we develop an iterated residue formula for tautological integrals on $\tilde{\calx_{k}}$, that is, integrals of the form $\int_{\tilde{\calx}_k}P(u,h)$
where $P$ is a homogeneous polynomial of degree $n+k(n-1)$, the dimension of $\tilde{\calx}_k$. This formula is based on a two-step localisation process (what we call the Snowman Model in \cite{bsz}) leading to a vanishing theorem of residues. We improve this model in this paper in order to apply it to tautological integrals.  

\subsection{Equivariant cohomology and localisation}\label{sec:equiv}

This section is a brief introduction to equivariant cohomology and localisation. For
more details, we refer the reader to \cite{bgv,bsz}. 

Let $\kt\cong U(1)^n$ be the maximal compact subgroup of
$T\cong(\CC^*)^n$, and denote by $\mathfrak{t}$ the Lie algebra of $\kt$.  
Identifying $T$ with the group $\CC^n$, we obtain a canonical basis of the weights of $T$:
$\lambda_1,\ldots ,\lambda_n\in\mathfrak{t}^*$. 

For a manifold $M$ endowed with the action of $\kt$, one can define a
differential $d_\kt$ on the space $S^\bullet \mathfrak{t}^*\otimes
\Omega^\bullet(M)^\kt$ of polynomial functions on $\mathfrak{t}$ with values
in $\kt$-invariant differential forms by the formula:
\[   
[d_\kt\alpha](X) = d(\alpha(X))-\iota(X_M)[\alpha(X)],
\]
where $X\in\mathfrak{t}$, and $\iota(X_M)$ is contraction by the corresponding
vector field on $M$. A homogeneous polynomial of degree $d$ with
values in $r$-forms is placed in degree $2d+r$, and then $d_\kt$ is an
operator of degree 1.  The cohomology of this complex--the so-called equivariant de Rham complex, denoted by $H^\bullet_T(M)$, is called the $T$-equivariant cohomology of $M$. Elements of $H_T^\bullet (M)$ are therefore polynomial functions $\mathfrak{t} \to \Omega^\bullet(M)^K$ and there is an integration (or push-forward map) $\int: H_T^\bullet(M) \to H_T^\bullet(\mathrm{point})=S^\bullet \mathfrak{t}^*$ defined as  
\[(\int_M \alpha)(X)=\int_M \alpha^{[\mathrm{dim}(M)]}(X) \text{ for all } X\in \mathfrak{t}\]
where $\alpha^{[\mathrm{dim}(M)]}$ is the differential-form-top-degree part of $\alpha$. The following proposition is the Atiyah-Bott-Berline-Vergne localisation theorem in the form of \cite{bgv}, Theorem 7.11. 
\begin{theorem}[Atiyah-Bott \cite{ab}, Berline-Vergne \cite{BV}]\label{abbv} Suppose that $M$ is a compact manifold and $T$ is a complex torus acting smoothly on $M$, and the fixed point set $M^T$ of the $T$-action on M is finite. Then for any cohomology class $\a \in H_T^\bullet(M)$
\[\int_M \alpha=\sum_{f\in M^T}\frac{\a^{[0]}(f)}{\mathrm{Euler}^T(T_fM)}.\]
Here $\mathrm{Euler}^T(T_fM)$ is the $T$-equivariant Euler class of the tangent space $T_fM$, and $\alpha^{[0]}$ is the differential-form-degree-0 part of $\alpha$. 
\end{theorem}

The right hand side in the localisation formula considered in the fraction field of the polynomial ring of $H_T^\bullet (\mathrm{point})=H^\bullet(BT)=S^\bullet \mathfrak{t}^*$ (see more on details in \cite{ab,bgv}). Part of the statement is that the denominators cancel when the sum is simplified.

\subsection{Equivariant Poincar\'e duals and multidegrees}
\label{subsec:epdmult} 

Restricting the equivariant de Rham complex to compactly supported (or quickly
decreasing at infinity) differential forms, one obtains the compactly
supported equivariant cohomology groups $ H^\bullet_{\kt,\mathrm{cpt}}(M)
$. Clearly $H^\bullet_{\kt,\mathrm{cpt}}(M) $ is a module over
$H^\bullet_\kt(M)$. For the case when $M=W$ is an $N$-dimensional
complex vector space, and the action is linear, one has
$H^\bullet_\kt(W)= S^\bullet\mathfrak{t}^*$ and $ H^\bullet_{\kt,\mathrm{cpt}}(W) $ is
a free module over $H^\bullet_\kt(W)$ generated by a single element of
degree $2N$:
\begin{equation}
  \label{thomg}
   H^\bullet_{\kt,\mathrm{cpt}}(W) = H^\bullet_{\kt}(W)\cdot\mathrm{Thom}_{\kt}(W),
\end{equation}

Fixing coordinates $y_1,\dots,y_N$ on $W$, in which the $T$-action is
diagonal with weights $\eta_1,\ldots,  \eta_N$, one can write an explicit
representative of  $\mathrm{Thom}_{\kt}(W)$ as follows:
\[   \mathrm{Thom}_{\kt}(W) = 
e^{-\sum_{i=1}^N|y_i|^2}\sum_{\sigma\subset\{1,\ldots , N\}}
\prod_{i\in\sigma}\eta_i/2\cdot\prod_{i\notin \sigma}dy_i\,d\bar y_i
\]

We will say that an algebraic variety has dimension $d$ if its
maximal-dimensional irreducible components are of dimension $d$.  A
$T$-invariant algebraic subvariety $\Sigma$ of dimension $d$ in $W$
represents $\kt$-equivariant $2d$-cycle in the sense that
\begin{itemize}
\item a compactly-supported equivariant form $\mu$ of degree $2d$ is
  absolutely integrable over the components of maximal dimension of
  $\Sigma$, and $\int_\Sigma\mu\in S^\bullet \mathfrak{t}$;
\item if $d_\kt\mu=0$, then $\int_\Sigma\mu$ depends only on the class
  of $\mu$ in $ H^\bullet_{\kt,\mathrm{cpt}}(W) $,
\item and $\int_\Sigma\mu=0$  if $\mu=d_\kt\nu$ for a
  compactly-supported equivariant form $\nu$.
\end{itemize}

\begin{definition} \label{defepd} Let $\Sigma$ be an $T$-invariant algebraic
  subvariety of dimension $d$ in the vector space $W$. Then the
  equivariant Poincar\'e dual of $\Sigma$ is the polynomial on $\mathfrak{t}$
  defined by the integral
\begin{equation}
 \label{vergneepd}
 \epd\Sigma = \frac1{(2\pi)^d}\int_\Sigma\mathrm{Thom}_{\kt}(W).
\end{equation}  
\end{definition}
\begin{remark}
  \begin{enumerate}
  \item An immediate consequence of the definition is that for an equivariantly
closed differential form $\mu$ with compact support, we have
\[  \int_\Sigma\mu = \int_W \epd\Sigma\cdot\mu.
\]
This formula serves as the motivation for the term {\em equivariant
  Poincar\'e dual.}
\item This definition naturally extends to the case of an analytic
  subvariety of $\CC^n$  defined in the neighborhood of the origin, or
  more generally, to any $T$-invariant cycle in $\CC^n$.
  \end{enumerate}
\end{remark}

We list some basic properties of the equivariant Poincar\'e dual. The
proofs can be found in \cite{rossmann},\cite{voj},\cite{milsturm}.
(cf. Proposition \ref{mdepd})
\begin{proposition}
\label{epdprops}
  \begin{description}
  \item[Positivity] The equivariant Poincar\'e dual $\epd\Sigma$ of a
    $d$-dimensional subvariety of $W$ is a homogeneous polynomial of
    degree $N-d$ in $\lambda_1,\ldots ,\lambda_n$, which may be expressed
    as a positive integral polynomial of the weights  $\eta_i$, $i=1,\ldots , N$.
  \item[Additivity] If $\Sigma_1,\Sigma_2\subset W$ are two
    $T$-invariant subvarieties of dimension $d$ having no common
    components of top dimension, then
    $\epd{\Sigma_1\cup\Sigma_2}=\epd{\Sigma_1}+\epd{\Sigma_2}$.
  \item[Deformation invariance] If $\Sigma_t$ is a flat algebraic
    family of varieties then $\epd{\Sigma_t}$ is independent of
    $t$. 
  \item[Symmetry] 
  Let $T=(\CC^*)^n$ be the subgroup of diagonal matrices of the complex
  group $\mathrm{GL}_n$, and denote by $\lambda_1,\ldots ,\lambda_n$ its basic
  weights. If $\Sigma$ is a $\mathrm{GL})_n$-invariant subvariety of the 
  $\mathrm{GL}_n$-module $W$, then the equivariant Poincar\'e dual $\epd{\Sigma,W}_T$ is
  a {\em symmetric} polynomial in $\lambda_1,\ldots ,\lambda_n$.
  \item[Complete intersections]  Let the variety $\Sigma$ be a complete
    intersection defined by $r$ relations:
    $f_1,\dots,f_r\in\CC[y_1,\dots,y_N]$ of degrees
    $\alpha_1,\dots,\alpha_r\in\mathfrak{t}^*$ correspondingly. Then 
    \begin{equation*}
      \label{ci}
   \epd{\Sigma} = \prod_{i=1}^r\alpha_i,      
    \end{equation*}

\item[Elimination] Let $\Sigma\subset W$ be a closed $T$-invariant
  subvariety, and denote by $I(\Sigma)$ the ideal of functions
  vanishing on $\Sigma$. Fix a polynomial $f\in \CC[y_1,\dots,y_N]$ of
  weight $\eta_0$, and let $\Sigma_f$ be the variety in
  $W\oplus\CC y_0$ with ideal generated by $I(\Sigma)$ and
  $y_0-f$. Then
\[\epd{\Sigma_f,W\oplus\CC y_0}=\eta_0\cdot \epd{\Sigma,W}\]
  \end{description}
\end{proposition}

\begin{remark}
  Another way of writing the formula for complete intersections is the
  following. Let $E$ be a $T$-vector space with a list of weights
  $\alpha_1,\ldots ,\alpha_r$, and denote by $\mathrm{Euler}^T(E)$ the equivariant
  Euler class of $E$, i.e.
  \begin{equation*}
     \mathrm{Euler}^T(E) = \prod_{i=1}^r\alpha_i.    
  \end{equation*}
Suppose that $\gamma:W\to E$ is an equivariant polynomial map with the
property that the differential $d\gamma:W\to E$ is surjective on a
Zariski open part of $\gamma^{-1}(0)$. Then 
\begin{equation*}
  \epd{\gamma^{-1}(0),W}=\mathrm{Euler}^T(E).
\end{equation*}
\end{remark}
\begin{remark} \label{passtoeuler} An important special case of complete
  intersections are the linear subspaces. For these, the formula \eqref{ci}
  takes the following form: for every subset $\mathbf{i}\subset\{1,\ldots ,
  N\}$ we have
\begin{equation}
\epd{
\{y_{i}=0,\,i\in\mathbf{i}\},W}=
\prod_{i\in\mathbf{i}}\eta_i.
\end{equation}
\end{remark}

Another incarnation of the equivariant Poincar\'e dual is the notion
of {\em multidegree}, which is close in spirit to the original
construction of Joseph \cite{joseph}. 

Introduce the notation $S=\CC[y_1,\ldots, y_N]$ for the polynomial
functions on $W$, and denote the ideal of the functions vanishing on the
$T$-invariant subvariety $\Sigma\subset W$ by $I(\Sigma)$; thus
$I(\Sigma)=\{f\in\CC[y_1,\ldots , y_N];\;f(p)=0\text{ if }p\in\Sigma\}$.

Consider a finite (length-$M$), $T$-graded resolution of $S/I(\Sigma)$
by free $S$-modules:
\[
\oplus_{i=1}^{j[M]}  Sw_i[M]\to\dots\to  \oplus_{i=1}^{j[m]}
 Sw_i[m]\to\dots\to \oplus_{i=1}^{j[1]} S w_i[1] \to S \to S/I(\Sigma)\to 0;
\]
where $w_i[m]$ is a free generator of degree $\eta_i[m]\in\Lambda$
for $i=1,\dots j[m]$, $m=1,\ldots , M$. Then the multidegree of the ideal
$I(\Sigma)$ is defined by the formula
\begin{equation*}
  \mdeg{I,S}_T = \frac1{D!}\sum_{m=1}^M\sum_{i=1}^{j[m]} (-1)^{D-m} \eta_i[m]^D
\end{equation*}
where $D$ is the codimension of $\Sigma$.

\begin{proposition}[\cite{rossmann}]
\label{mdepd} Let  $\Sigma \subset W$ be a $T$-invariant
subvariety. Then we have
\[       \epd{\Sigma,W}_T=\mdeg{I(\Sigma),\CC[y_1,\ldots , y_N]}.
\]
\end{proposition}
Following \cite{milsturm} \S8.5, now we sketch an algorithm for
computing $\mdeg{I,S}$, proving that the axioms determine this
invariant.
The monomials $\mathbf{y}^{\mathbf{a}}=\prod_{i=1}^Ny_i^{a_i}\in
S=\CC[y_1,\ldots, y_N]$ are parametrized by the integer vectors
$\mathbf{a}=(a_1\ldots  a_N)\in\ZZ_+^N$. A {\em monomial order} $<$ on
$S$ is a total order of the monomials in $S$ such that for any three
monomials $m_1,m_2,n$ satisfying $m_1>m_2$, we have $nm_1>nm_2>m_2$
(see \cite[\S 15.2]{eisenbud} ). An ordering of the coordinates $y_1,\ldots, y_N$ induces the
so-called {\em  lexicographic} monomial order of the monomials, that
is, $\mathbf{y}^{\mathbf{a}}>\mathbf{y}^{\mathbf{b}}$ if and only if
$a_i>b_i$ for the first index $i$ with $a_i \neq b_i$. 

Let $I\subset S$ be a $T$-invariant ideal. Define the {\em initial
  ideal} $\mathrm{in}_<(I)\subset S$ to be the ideal generated by
the monomials $\left\{\mathrm{in}_<(p):p\in I \right\}$, where
$\mathrm{in}_<(p)$ is the largest monomial of $p$ w.r.t $<$. There is
a flat deformation of $I$ into $\mathrm{in}_<(I)$ (see \cite{eisenbud},
Theorem 15.17.).

An ideal $M\subset S$ generated by a set of monomials in
$y_1,\ldots, y_N$ is called a \emph{monomial ideal}. Since
$\mathrm{in}_<(I)$ is such an ideal, by the deformation invariance
it is enough to compute $\mdeg{M}$ for monomial ideals $M$. If the
codimension of $\Sigma(M)$ in $W$ is $s$, then the maximal
dimensional components of $\Sigma(M)$ are codimension-$s$ coordinate
subspaces of $W$. Such subspaces are indexed by subsets
$\mathbf{i}\in\{1\ldots  N\}$ of cardinality $s$; the corresponding
associated primes are $\mathfrak{p}[\mathbf{i}]=\langle y_i:i\in
\mathbf{i} \rangle$. Then 
\begin{equation*}
\label{primemon} \mult(\mathfrak{p}[\mathbf{i}],M)=
\left|\left\{\mathbf{a}\in\ZZ_+^{[{\mathbf{i}}]};\;
\mathbf{y}^{\mathbf{a}+\mathbf{b}}\notin M\text{ for all }
\mathbf{b}\in\ZZ_+^{[\hat{\mathbf{i}}]}\right\}\right|,
\end{equation*}
where $\ZZ_+^{[\mathbf{i}]}=\{\mathbf{a}\in \ZZ_+^N;a_i=0 \text{ for }
i\notin \mathbf{i}\}$, $\hat{\mathbf{i}}=\{1\ldots 
N\}\setminus\mathbf{i}$, and $|\cdot|$, as usual, stands for the
number of elements of a finite set. By the normalization and additivity axiom we have
\begin{equation}\label{mdegformula}
\mdeg{M,S} =
\sum_{|\mathbf{i}|=s}\mult(\mathfrak{p}[\mathbf{i}],M)
\prod_{i\in\mathbf{i}}\eta_i.
\end{equation}
By definition, the weights $\eta_1,\ldots \eta_N$ on $W$ are linear forms of $\l_1,\ldots \l_r$, the basis of $(\CC^*)^r$, and we denote the coefficient of $\l_j$ in $\eta_i$ by $\coeff(\eta_i,j)$, $1\le i\le N, 1\le j \le r$, and introduce 
\[\deg(\eta_1,\ldots, \eta_N;m)=\#\{i;\;\coeff(\eta_i,m)\neq 0\}\}.\]
It is clear from the formula \eqref{mdegformula} that 
\begin{equation*}
\deg_{\l_m}\mdeg{I,S} \le \deg(\eta_1,\ldots, \eta_N;m)
\end{equation*}
holds for any $1\le m \le r$. 
We need a slightly stronger result in the next section which we formulate and prove here.

\begin{proposition}\label{vanishlemma}
Let $W$ be an $N$-dimensional complex
vector space with coordinates $y_1,\dots,y_N$ endowed with an diagonal action of $(\CC^*)^r$ acting with weights $\eta_1\ldots \eta_N$. Let $I\subset S$ be a $(\CC^*)^r$-invariant ideal. Then 
\begin{equation*}
\deg_{\l_m}\mdeg{I,S} \le \deg(\eta_1,\ldots, \eta_N;m)-1
\end{equation*}
\end{proposition}

\proof
By the positivity property of the multidegree $\mdeg{I,S}$ is indeed a polynomial of the weights $\eta_i,i=1,\ldots, N$. Let 
\[\coeff(\eta_i,m)\neq 0 \text{ for } 1 \le i \le s;\ \coeff(\eta_{s+1},m)=\ldots =\coeff(\eta_{N},m)=0.\]
The idea of the proof is to choose an appropriate monomial order on the polynomial ring $S=\CC[y_1,\ldots, y_N]$ to ensure that $y_1$ does not appear in the corresponding initial ideal. 

To that end recall, that a weight function is a linear map $\rho: \ZZ^N \to \ZZ$. This defines a partial order $>_\rho$ on the monomials of $S$, called the weight order associated to $\rho$, by the rule $m=y^a >_\rho n=y^b$ iff $\rho(a)>\rho(b)$. Here $a=(a_1,\ldots ,a_N),b=(b_1,\ldots ,b_N)$ are arbitrary multiindices. Any weight order can be extended to a compatible monomial order $>$ (see \cite{eisenbud}, Ch 15.2), which means that $m>_\rho n$ implies $m>n$. 
For our purposes define 
\begin{equation*}
\rho(y_1)=-1, \rho(y_2)=\ldots =\rho(y_N)=0
\end{equation*}    
and let $>$ denote arbitrary compatible monomial order on $S$. By definition for a monomial $m \in S$
\begin{equation}\label{weights}
\rho(m)<0 \Longleftrightarrow y_1|m
\end{equation} 
Let $p \in I$, and assume that not all monomials of $p$ are divisible by $y_1$. If they all did, $y_1|p$, and therefore $in_{>}(p/y_1)|in_{>}(p)$ would hold, and therefore $p$ would not be among the generators of the $in_>(I)$. Therefore $y_1$ does not divide $p$.

Then there is a monomial of $p$ not containing $y_1$, and by \eqref{weights} the weight of this monomial is strictly bigger to the weight of any other containing $y_1$.  
Consequently, $y_1$ does not divide any of the generators of $in_>(I)$, and by \eqref{mdegformula} $\mdeg{I,S}$ does not depend on $\eta_1$. The only possible variables containing $\l_m$ are therefore $\eta_2,\ldots ,\eta_s$, giving a maximum total degre $s-1$.  
\qed 
\subsection{Equivariant localisation on $\tilde{\calx}_{k}$}\label{subsec:loc}
In this subsection we develop a two step equivariant localisation method on $\tilde{\calx}_{k}$ which is a fibred and stronger version of our iterated residue in \cite{bsz}. It is based on the Rossmann equivariant localisation formula, which is an improved version of the Atiyah-Bott/Berline-Vergne localisation for singular varieties sitting in a smooth ambient space. 

We also refer this later as the Snowman Model, due to the figure in $\S 6$ of \cite{bsz}, which summarises the process. We need an important restriction on the parameters to make this method work, namely we assume that $k\le n$ in this section.
In $\S$\ref{quotient} we defined a partial resolution $\tilde{X}_k \to X_k$, which fibres over the flag manifold $\flag_k(\CC^n)$
\begin{equation*}
\xymatrix{\tilde{X}_k \ar[r]^-{\rho} \ar[d]^{\mu} & X_k \subset \flag_k(\symdot) \\
\Hom(\CC^k,\CC^n)/B_k=\flag_k(\CC^n)& } 
\end{equation*}
where the fibres of $\mu$ are isomorphic to $\overline{P_{k,n}\cdot p_n}\subset \flag_k(\symdot)$.
The fibred version of this diagram is a double fibration
\begin{equation}\label{diagram}
\xymatrix{\tilde{\calx}_k \ar[r]^-{\rho} \ar[d]^{\mu} & \calx_k\subset \flag_k(\symdotx) \\
\flag_k(T_X) \ar[d]^{\tau} & \\
X & } 
\end{equation}
where $\flag_k(T_X)$ is the flag bundle of $T_X$, and over every point $x\in X$ we get back the previous diagram. 

Let $e_1,\ldots, e_n \in \CC^n$ be an eigenbasis of $\CC^n$ for the $T$ action with weights $\l_1,\ldots, \l_n\in \mathfrak{t}^*$ and let 
\[\ff=(\langle e_1 \rangle \subset \langle e_1,e_2 \rangle \subset \ldots \subset \langle e_1,\ldots,e_k \rangle \subset \CC^n)\]
denote the standard flag in $\CC^n$ fixed by the upper Borel.

Recall our notations \eqref{uhnotations} for the canonical line bundles on $\tilde{\calx}_k$ and $X \subset \PP^{n+1}$.
The fibres of $\tilde{\pi}:\tilde{\calx}_k \to X$ are canonically isomorphic to $\tilde{X}_k$. Localisation on $\tilde{X}_k$ has been worked out in \cite{bsz}, here we adapt and improve this method for our purposes. 

Since $\tilde{X}_{k}$ fibres over the flag manifold $\flag_k(\CC^n)$,  Proposition \ref{abbv} gives us   
\begin{equation} \label{flagloc}
\int_{\tilde{X}_k}\alpha= \sum_{\sigma\in\sg n/\sg{n-k}}
\frac{\alpha_{\sigma(\ff)}}{\prod_{1\leq m \leq
k}\prod_{i=m+1}^n(\lambda_{\sigma\cdot
    i}-\lambda_{\sigma\cdot m})},
\end{equation}
where 
\begin{itemize}
\item $\sigma$ runs over the ordered $k$-element subsets of $\{1,\ldots, n\}$ labeling the fixed flags $\sigma(\ff)=(\langle e_{\sigma(1)} \rangle \subset \ldots \subset \langle e_{\sigma(1)},\ldots, e_{\sigma(k)} \rangle \subset \CC^n)$ in $\CC^n$, 
\item $\prod_{1\leq m \leq k}\prod_{i=m+1}^n(\lambda_{\sigma(i)}-\lambda_{\sigma(m)})$ is the equivariant Euler class of the tangent space of $\flag_k(\CC^n)$ at $\s(\ff)$,
\item if $X_{\sigma(\ff)}=\mu^{-1}(\sigma(\ff))$ denotes the fibre then $\alpha_{\sigma(\ff)}=(\int_{X_{\sigma(\ff)}} \alpha)^{[0]}(\sigma(\ff))\in S^\bullet \mathfrak{t}^*$ is the differential-form-degree-zero part evaluated at $\sigma(\ff)$.
\end{itemize}
In particular, when $\alpha=\alpha(u)$ is a polynomial of $u=c_1(\calo_{\tilde{X}_k}(1)$, then $u$ is represented by $\l_{\s(1)}+\ldots +\l_{\s(k)}\in \mathfrak{t}^*$ at the fixed point $\sigma(\ff)$, and therefore 
\begin{equation}\label{alphasigmaf}
\alpha_{\s(\ff)}=\sigma \cdot \alpha_\ff=\alpha_\ff(\l_{\s(1)}+\ldots+\l_{\s(k)})\in S^\bullet \mathfrak{t}^*,
\end{equation}
is the $\sigma$-shift of the polynomial $\alpha_{\ff}=(\int_{X_{\ff}}\alpha)^{[0]}(\ff)\in S^\bullet \mathfrak{t}^*$ corresponding to the distinguished fixed flag $\ff$.

\subsection{Proof of Theorem \ref{maintechnical}: Transforming the localisation formula into iterated residue}\label{subsec:transform}

In this section we transform the right hand side of \eqref{flagloc} into an iterated residue motivated by \cite{bsz}. This step turns out to be crucial in handling the combinatorial complexity of the fixed point data in the Atiyah-Bott localisation formula and condense the symmetry of this fixed point data in an efficient way which enables us to prove the vanishing of the contribution of all but one of the fixed points. 

To describe this formula, we will need the notion of an {\em iterated
  residue} (cf. e.g. \cite{szenes}) at infinity.  Let
$\omega_1,\dots,\omega_N$ be affine linear forms on $\CC^k$; denoting
the coordinates by $z_1,\ldots, z_k$, this means that we can write
$\omega_i=a_i^0+a_i^1z_1+\ldots + a_i^kz_k$. We will use the shorthand
$h(\bz)$ for a function $h(z_1\ldots z_k)$, and $\dbz$ for the
holomorphic $n$-form $dz_1\wedge\dots\wedge dz_k$. Now, let $h(\bz)$
be an entire function, and define the {\em iterated residue at infinity}
as follows:
\begin{equation}
  \label{defresinf}
 \ires \frac{h(\bz)\,\dbz}{\prod_{i=1}^N\omega_i}
  \overset{\mathrm{def}}=\left(\frac1{2\pi i}\right)^k
\int_{|z_1|=R_1}\ldots
\int_{|z_k|=R_k}\frac{h(\bz)\,\dbz}{\prod_{i=1}^N\omega_i},
 \end{equation}
 where $1\ll R_1 \ll \ldots \ll R_k$. The torus $\{|z_m|=R_m;\;m=1 \ldots
 k\}$ is oriented in such a way that $\res_{z_1=\infty}\ldots
 \res_{z_k=\infty}\dbz/(z_1\cdots z_k)=(-1)^k$.
We will also use the following simplified notation: $\sires \overset{\mathrm{def}}=\ires.$

In practice, one way to compute the iterated residue \eqref{defresinf} is the following algorithm: for each $i$, use the expansion
 \begin{equation}
   \label{omegaexp}
 \frac1{\omega_i}=\sum_{j=0}^\infty(-1)^j\frac{(a^{0}_i+a^1_iz_1+\ldots
   +a_{i}^{q(i)-1}z_{q(i)-1})^j}{(a_i^{q(i)}z_{q(i)})^{j+1}},
   \end{equation}
   where $q(i)$ is the largest value of $m$ for which $a_i^m\neq0$,
   then multiply the product of these expressions with $(-1)^kh(z_1\ldots
   z_k)$, and then take the coefficient of $z_1^{-1} \ldots z_k^{-1}$
   in the resulting Laurent series.

We repeat the proof of the following iterated residue theorem from \cite{bsz}.
\begin{proposition}{[\cite{bsz} Proposition 5.4]} For any homogeneous polynomial $Q(\bz)$ on $\CC^k$ we have
\begin{equation}\label{flagres}
\sum_{\sigma\in\sg n/\sg{n-k}}
\frac{Q(\lambda_{\sigma(1)},\ldots ,\lambda_{\sigma(k)})}
{\prod_{1\leq m\leq k}\prod_{i=m+1}^n(\lambda_{\sigma\cdot
    i}-\lambda_{\sigma\cdot m})}=\sires
\frac{\prod_{1\leq m<l\leq k}(z_m-z_l)\,Q(\bz)\dbz}
{\prod_{l=1}^k\prod_{i=1}^n(\lambda_i-z_l)}
\end{equation}
\end{proposition}

\begin{proof}
  We compute the iterated residue \eqref{flagres} using the Residue
  Theorem on the projective line $\CC\cup\{\infty\}$.  The first
  residue, which is taken with respect to $z_k$, is a contour
  integral, whose value is minus the sum of the $z_k$-residues of the
  form in \eqref{flagres}. These poles are at $z_k=\lambda_j$,
  $j=1\ldots n$, and after canceling the signs that arise, we obtain the
  following expression for the right hand side of  \eqref{flagres}:
\[
\sum_{j=1}^n \frac{\prod_{1\leq m<l\leq
    k-1}(z_m-z_l)\,\prod_{l=1}^{k-1}(z_l-\lambda_j)\,Q(z_1\ldots
  z_{k-1},\lambda_j)\;dz_1\dots dz_{k-1}}
{\prod_{l=1}^{k-1}\prod_{i=1}^n(\lambda_i-z_l)\prod_{i\neq
    j}^n(\lambda_i-\lambda_j)}.
\]
After cancellation and exchanging the sum and the residue operation,
at the next step, we have
\[
(-1)^{k-1}\sum_{j=1}^n\res_{z_{k-1}=\infty} \frac{\prod_{1\leq
m<l\leq
    k-1}(z_m-z_l)\,Q(z_1\ldots z_{k-1},\lambda_j)\;dz_1\dots
  dz_{k-1}} {\prod_{i\neq j}^n
  \left((\lambda_i-\lambda_j)\prod_{l=1}^{k-1}(\lambda_i-z_l)\right)}.
\]
Now we again apply the Residue Theorem, with the only difference that
now the pole $z_{k-1}=\lambda_j$ has been eliminated. As a result,
after converting the second residue to a sum, we obtain
\[
(-1)^{2k-3}\sum_{j=1}^n\sum_{s=1,\,s\neq j}^n \frac{\prod_{1\leq
m<l\leq
    k-2}(z_l-z_m)\,Q(z_1\ldots
  z_{k-2},\lambda_s,\lambda_j)\;dz_1\dots dz_{k-2}}
{(\lambda_s-\lambda_j)\prod_{i\neq j,s}^n
  \left((\lambda_i-\lambda_j)(\lambda_i-\lambda_s)\prod_{l=1}^{k-1}(\lambda_i-z_l)\right)}.
\]
Iterating this process, we arrive at a sum very similar to
(\ref{flagloc}). The difference between the two sums will be the
sign: $(-1)^{k(k-1)/2}$, and that the $k(k-1)/2$ factors of the form
$(\lambda_{\sigma(i)}-\lambda_{\sigma(m)})$ with $1\le m<i\le k$ in
the denominator will have opposite signs. These two differences
cancel each other, and this completes the proof.
\end{proof}
\begin{remark}
  Changing the order of the variables in iterated residues, usually,
  changes the result. In this case, however, because all the poles are
  normal crossing, formula \eqref{flagres} remains true no matter in
  what order we take the iterated residues.
\end{remark}

This together with \eqref{flagloc} and \eqref{alphasigmaf} gives
\begin{proposition}\label{propflag} Let $k\le n$ and $\alpha(u)$ a polynomial in $u=c_1(\calo_{\tilde{X}_k}(1))$. Then 
\begin{equation*}
\int_{\tilde{X}_k}\alpha(u)=\sires
\frac{\prod_{1\leq m<l\leq k}(z_m-z_l)\,\alpha_{\ff}(z_1+\ldots +z_k)\dbz}
{\prod_{l=1}^k\prod_{i=1}^n(\lambda_i-z_l)}
\end{equation*}
\end{proposition}
Following \cite{bsz}, we proceed a second localisation on the fibre 
\[X_{\ff}=\pi^{-1}(\ff)\simeq \overline{P_{k,n}\cdot p_k}\subset \flag_k(\symdot)\]
 to compute $\alpha_\ff(z_1+\ldots +z_k)$. Since $X_{\ff}$
 is invariant under the $T$-action on $\flag_k(\symdot)$, we can apply Rossmann's integration formula, which is explained in $\S 3.1$ of \cite{bsz}, but we sketch the statement here again.  

Let $Z$ be a complex manifold with a holomorphic $T$-action, and let
$M\subset Z$ be a $T$-invariant analytic subvariety with an isolated
fixed point $p\in M^T$. Then one can find local analytic coordinates
near $p$, in which the action is linear and diagonal. Using these
coordinates, one can identify a neighborhood of the origin in $\TT_pZ$
with a neighborhood of $p$ in $Z$. We denote by $\tc_pM$ the part of
$\TT_pZ$ which corresponds to $M$ under this identification;
informally, we will call $\tc_pM$ the $T$-invariant {\em tangent cone}
of $M$ at $p$. This tangent cone is not quite canonical: it depends on
the choice of coordinates; the multidegree of
$\Sigma=\tc_pM$ in $W=\TT_pZ$, however, does not. Rossmann named this
 the {\em equivariant multiplicity of $M$ in $Z$ at $p$}:
\begin{equation}\label{emult}
   \emu_p[M,Z] \overset{\mathrm{def}}= \mdeg{\tc_pM,\TT_pZ}.
\end{equation}

\begin{remark}
In the algebraic framework one might need to pass to the {\em
tangent scheme} of $M$ at $p$ (cf. \cite{fulton}). This is canonically
defined, but we will not use this notion.
\end{remark}
\begin{proposition}[Rossmann's localisation formula \cite{rossmann}] Let $\mu \in H_T^*(Z)$ be an equivariant class represented by a holomorphic equivariant map $\mathfrak{t} \to\Omega^\bullet(Z)$. Then 
\begin{equation}
  \label{rossform}
  \int_M\mu=\sum_{p\in M^T}\frac{\emu_p[M,Z]}{\mathrm{Euler}^T(\TT_pZ)}\cdot\mu^{[0]}(p),
\end{equation}
where $\mu^{[0]}(p)$ is the differential-form-degree-zero component
of $\mu$ evaluated at $p$.  
\end{proposition}
In \cite{bsz} we apply this formula with $M=X_\ff, Z=\flag_k^*(\symdot)$ and $\mu=$\\$\mathrm{Thom}(\flag_k^*)$, the equivariant Thom class of $\flag_k^*(\symdot)$ where 
\[\flag_k^*(\symdot)=\{V_1 \subset \ldots \subset V_k\subset \symdot: \dim(V_i)=i, V_i\subset \mathrm{Span}_\CC(e_\tau:\Sigma \tau \le i)\}\]
is a submanifold of $\flag_k(\symdot)$.

Here we apply the Rossman formula for $M=X_\ff, Z=\flag_k^*(\symdot)$ and $\mu=\alpha_\ff$. The fixed points on $Z=\flag_k^*(\symdot)$ are parametrized by {\it admissible} sequences of partitions $\bipi=(\pi_1,\ldots, \pi_k)$. We call a sequence of partitions $\bipi=(\pi_1 \ldots \pi_k)\in\Pi^{\times d}$ admissible if
\begin{enumerate}
\item $\Sigma \pi_l\le l$ for $1\le l \le k$, and 
\item $\pi_l\neq\pi_m$ for $1\leq l\neq m\leq k$. 
\end{enumerate}
We will denote the set of admissible sequences of length $k$ by $\Bipi$. Then \eqref{rossform} and Proposition \ref{propflag} give us (see \cite{bsz}) 
\begin{proposition}\label{propint} Let $k\le n$ and $\alpha(u)$ a polynomial in $u=c_1(\calo_{\tilde{X}_k}(1))$. Then 
\begin{equation}\label{intnumberone} 
\int_{\tilde{X}_k}\alpha=\sum_{\bipi\in\Bipi} \sires \frac{
Q_\bipi(\bz)\,\prod_{m<l}(z_m-z_l) \alpha(z_{\pi_1}+\ldots +z_{\pi_k})}{
\prod_{l=1}^k\prod_{\tau\leq l}^{\tau\neq\pi_1\ldots \pi_l}
(z_{\tau}-z_{\pi_l})  \prod_{l=1}^k\prod_{i=1}^n(\lambda_i-z_l)} \,\dbz.
  \end{equation}
where $Q_\bipi(\bz)=\emu_\bipi[X_\ff,\flag_k^*]$ and $z_\pi=\sum_{i\in \pi}z_i$.
\end{proposition}  
The following theorem is a stronger version of the main theorem Proposition 6.1 in \cite{bsz} for tautological integrals. We devote the next section to the proof. 
\begin{theorem}[\textbf{The Residue Vanishing Theorem}]\label{vanishtheorem} Let $k\le n$ and let $\alpha \in \Omega^{k(n-1)}(\tilde{X}_k)$ be a top from. Then 
\begin{enumerate}
\item All terms but the one corresponding to $\bipi_\dist=([1],[2],\ldots, [k])$ vanish in \eqref{intnumberone} leaving us with
\begin{equation} \label{intnumberoneandhalf}
\int_{\tilde{X}_k}\alpha=\sires \frac{
Q_{[1],\ldots, [k]}(\bz)\,\prod_{m<l}(z_m-z_l) \alpha(z_1+\ldots +z_k,h)}{
\prod_{\um \tau \le l \le k}
(z_\tau-z_l)  \prod_{l=1}^k\prod_{i=1}^n(\lambda_i-z_l)} \,\dbz.
  \end{equation} 
 \item If $|\tau|\ge 3$ then $Q_{([1],\ldots, [k])}(\bz)$ is divisible by $z_\tau-z_l$ for all $l \ge \um \tau$, so we arrive at the simplified formula
\begin{equation} \label{intnumberthree}
\int_{\tilde{X}_k}\alpha=\sires \frac{
Q_k(\bz)\,\prod_{m<l}(z_m-z_l) \alpha(z_1+\ldots +z_k,h)}{
\prod_{m+r \le l \le k}
(z_m+z_r-z_l)  \prod_{l=1}^k\prod_{i=1}^n(\lambda_i-z_l)} \,\dbz.
  \end{equation}
\end{enumerate}
\end{theorem} 

\begin{remark}\label{remarkq}
\begin{enumerate}
\item 
The geometric meaning of $Q_k(\bz)$ in \eqref{intnumberthree} is the following, see \cite{bsz} Theorem 6.16. Let $T_k\subset B_k\subset \GL(k)$ be the subgroups of invertible
  diagonal and upper-triangular matrices, respectively; denote the
  diagonal weights of $T_k$ by $z_1, \ldots, z_k$.  Consider the $\GL(k)$-module of 3-tensors $\Hom(\CC^k,\sym^2\CC^k)$; identifying the
  weight-$(z_m+z_r-z_l)$ symbols $q^{mr}_l$ and $q^{rm}_l$, we can
  write a basis for this space as follows:
\[ \Hom(\CC^k,\sym^2\CC^k)=\bigoplus \CC q^{mr}_l,\;  1\leq m,r,l \leq k.
\]
Consider the point $\epsilon=\sum_{m=1}^k\sum_{r=1}^{k-m}q_{mr}^{m+r}$
in  the $B_k$-invariant subspace
\begin{equation*}
  \label{nhmodule}
    N_k = \bigoplus_{1\leq m+r\leq l\leq k} \CC q^{mr}_l\subset
\Hom(\CC^k,\sym^2\CC^k).
\end{equation*}
Set the notation $\OO_k$ for the orbit closure
$\overline{B_k\epsilon}\subset N_k$, then $Q_k(\bz)$ is the $T_k$-equivariant
Poincar\'e dual $Q_k(\bz) = \epd{\OO_k,N_k}_{T_k}$,
which is a homogeneous polynomial of degree
$\dim(N_k)-\dim(\OO_k).$. For small $k$ these polynomials are the following (see \cite{bsz} $\S7$):
\[Q_2=Q_3=1, Q_4=2z_1+z_2-z_4, Q_5=2z_1^2 +3z_1z_2-2z_1z_5+2z_2z_3-z_2z_4-z_2z_5-z_3z_4+z_4z_5.\] 
\item To understand the significance of this vanishing theorem we note that while the fixed point set $\Bipi$ on $\flag_k^*(\symdot)$ is well understood, it is not clear which of these fixed points sit in $X_{\mathbf{f}}$. But we have enough information to prove that none of those fixed points in $X_{\mathbf{f}}$ contribute to the iterated residue except for the distinguised fixed point $\bipi_\dist=([1],[2],\ldots, [k])$.
\item Theorem \ref{maintechnical} follows by substituting $\frac{1}{\prod_{i=1}^n(\l_i-z_j)}=\frac{1}{z_j^nc(1/z_j)}=\frac{s(1/z_j)}{z_j^n}$ for $ j=1,\ldots, k$ followed by integration over $X$ on both sides of \eqref{intnumberthree}.
\end{enumerate}
\end{remark}

\subsection{The vanishing of residues}
\label{subsec:vanres}

In this subsection following \cite{bsz} $\S6.2$ we describe the conditions under which iterated
residues  of the type appearing in the sum in \eqref{intnumberone}
vanish and we prove Theorem \ref{vanishtheorem}.

We start with the 1-dimensional case, where the residue at infinity
is defined by \eqref{defresinf} with $d=1$. By bounding the integral
representation along a contour $|z|=R$ with $R$ large, one can
easily prove
\begin{lemma}\label{1lemma}
  Let $p(z),q(z)$ be polynomials of one variable. Then
\[\res_{z=\infty}\frac{p(z)\,dz}{q(z)}=0\quad\text{if }\deg(p(z))+1<\deg(q).
\]
\end{lemma}
Consider now the multidimensional situation. Let $p(\bz),q(\bz)$ be
polynomials in the $k$ variables $z_1\ldots z_k$, and assume that
$q(\bz)$ is the product of linear factors $q=\prod_{i=1}^N L_i$, as
in \eqref{intnumberone}. We continue to use the notation $\dbz=dz_1\dots dz_k$.
We would like to formulate conditions under which the iterated
residue
\begin{equation}
  \label{ires}
\ires\frac{p(\bz)\,\dbz}{q(\bz)}
\end{equation}
vanishes. Introduce the following notation:
\begin{itemize}\label{notations}
\item For a set of indices $S\subset\{1\ldots k\}$, denote by $\deg(p(\bz);S)$
  the degree of the one-variable polynomial $p_S(t)$ obtained from $p$
  via the substitution $z_m\to
  \begin{cases}
t\text{ if }m\in S,\\ 1\text{ if }m\notin S.
  \end{cases}$
\item For a nonzero linear function $L=a_0+a_1z_1+\ldots +a_kz_k$,
  denote by $\coeff(L,z_l)$ the coefficient $a_l$;
\item finally, for $1\leq m\leq k$, set
  \[\lead(q(\bz);m)=\#\{i;\;\max\{l;\;\coeff(L_i,z_l)\neq0\}=m\},\]
which is the number of those factors $L_i$ in which the coefficient
of $z_m$ does not vanish, but the coefficients of $z_{m+1},\ldots,
z_k$ are $0$.
\end{itemize}
Thus we group the $N$ linear factors of $q(\bz)$ according to the
nonvanishing coefficient with the largest index; in particular, for
$1\leq m\leq k$ we have
\[   \deg(q(\bz);m)\geq\lead(q(\bz);m),\, \text{ and } \sum_{m=1}^k\lead(q(\bz);m)=N.
\]
Now applying Lemma \ref{1lemma} to the first residue in
\eqref{ires}, we see that
\[ \res_{z_k=\infty}\frac{p(z_1,\ldots,z_{k-1},z_k)\dbz}{q(z_1,\ldots ,z_{k-1},z_k)}=0
\]
whenever $\deg(p(\bz);k)+1<\deg(q(\bz),k)$; in this case, of course,
the entire iterated residue \eqref{ires} vanishes.

Now we suppose the residue with respect to $z_k$ does not vanish,
and we look for conditions of vanishing of the next residue:
\begin{equation}
  \label{2res}
\res_{z_{k-1}=\infty}\res_{z_k=\infty}\frac{p(z_1,\ldots,z_{k-2},z_{k-1},z_k)\dbz}
{q(z_1,\ldots,z_{k-2},z_{k-1},z_k)}.
\end{equation}
 Now the condition $\deg(p(\bz);k-1)+1<\deg(q(\bz),k-1)$ will
{\em insufficient};  for example,
\begin{equation}
\res_{z_{k-1}=\infty}\res_{z_k=\infty}\frac{kz_{k-1}kz_k}{z_{k-1}(z_{k-1}+z_k)}=
\res_{z_{k-1}=\infty}\res_{z_k=\infty}\frac{kz_{k-1}kz_k}{z_{k-1}z_k}
\left(1-\frac{z_{k-1}}{z_k}+\ldots\right)=1.
\end{equation}
After performing the expansions \eqref{omegaexp} to $1/q(\bz)$, we
obtain a Laurent series with terms $z_1^{-i_1}\ldots z_k^{-i_k}$ such that
$i_{k-1}+i_k\geq\mathrm{deg}(q(z);k-1,k)$, hence the condition
\begin{equation}\label{toprove}
\deg(p(\bz);k-1,k)+2<\deg(q(\bz);k-1,k)
\end{equation}
will suffice for the vanishing of \eqref{2res}.

There is another way to ensure the vanishing of \eqref{2res}: suppose
that for $i=1\ldots N$, every time we have $\coeff(L_i,z_{k-1})\neq0$,
we also have $\coeff(L_i,z_{k})=0$, which is equivalent to the
condition $\deg(q(\bz),k-1)=\lead(q(\bz);k-1)$.  Now the Laurent
series expansion of $1/q(\bz)$ will have terms $z_1^{-i_1}\ldots
z_k^{-i_k}$ satisfying $i_{k-1}\geq
\deg(q(\bz),k-1)=\lead(q(\bz);k-1)$, hence, in this case the vanishing
of \eqref{2res} is guaranteed by $
\deg(p(\bz),k-1)+1<\deg(q(\bz),k-1)$.  This argument easily
generalises to the following statement.
\begin{proposition}
  \label{vanishprop}
Let $p(\bz)$ and $q(\bz)$ be polynomials in the variables $z_1\ldots 
z_k$, and assume that $q(\bz)$ is a product of linear factors:
$q(\bz)=\prod_{i=1}^NL_i$; set $\dbz=dz_1\dots dz_k$. Then
\[ \ires\frac{p(\bz)\dbz}{q(\bz)} = 0
\]
if for some $l\leq k$, either of the following two options hold:
\begin{itemize}
\item $\deg(p(\bz);k,k-1,\dots,l)+k-l+1<\deg(q(\bz);k,k-1,\dots,l),$
\\ or
\item  $\deg(p(\bz);l)+1<\deg(q(\bz);l)=\lead(q(\bz);l)$.
\end{itemize}
\end{proposition}
Note that for the second option, the equality
$\deg(q(\bz);l)=\lead(q(\bz);l)$ means that
  \begin{equation}
    \label{op2cond}
\text{ for each }i=1\ldots N\text{ and }m>l,\,
\coeff(L_i,z_{l})\neq0\text{ implies }\coeff(L_i,z_{m})=0.
  \end{equation}

Recall that our goal is to show that all the terms of the sum in
\eqref{intnumberone} vanish except for the one corresponding to
$\bipi_\dist=([1]\ldots  [k])$. Let us apply our new-found tool,
Proposition \ref{vanishprop}, to the terms of this sum, and see what
happens.

Fix a sequence $\bipi=(\pi_1,\dots,\pi_k)\in\Bipi$, and consider the
iterated residue corresponding to it on the right hand side of
\eqref{intnumberone}. The expression under the residue is the product of
two fractions:
\[\frac{p(\bz)}{q(\bz)}=\frac{p_1(\bz)}{q_1(\bz)}\cdot\frac{p_2(\bz)}{q_2(\bz)},\]
where
\begin{equation}
  \label{pq}
\frac{p_1(\bz)}{q_1(\bz)}=\frac{
Q_\bipi(\bz)\,\prod_{m<l}(z_m-z_l) }{
\prod_{l=1}^k\prod_{\um\tau\leq
l}^{\tau\neq\pi_1\ldots \pi_l}(z_\tau-z_{\pi_{l}})}\text{ and
  }
\frac{p_2(\bz)}{q_2(\bz)} = \frac{
R(z_{\pi_1}+\ldots +z_{\pi_k},h,c_1)} {
\prod_{l=1}^k\prod_{i=1}^n(\lambda_i-z_l)}.
\end{equation}

Note that $p(\bz)$ is a polynomial, while $q(\bz)$ is a product of
linear forms.  

After these preparations we are ready to prove Theorem \ref{vanishtheorem}.
As a warm-up, we show that if the last element of the sequence is
not the trivial partition, i.e. if $\pi_k \neq [k]$, then already the
first residue in the corresponding term on the right hand side of
\eqref{intnumberone} -- the one with respect to $z_k$ -- vanishes.
Indeed, if $\pi_k\neq[k]$, then $\deg(q_2(\bz);k)=n$, while
$z_k$ does not appear in $p_2(\bz)$. 

On the other hand, $\deg(q_1(\bz);k)=1$, because the only term with $z_k$ is the one corresponding to $l=k, \tau=[k]\neq \pi_k$. If $\deg(Q_\bipi(\bz),k)=0$ held, we would be ready, as
\begin{equation}
\deg(p(\bz);k)=k-1 \text{ and } \deg(q(\bz);k)=n+1
\end{equation}
would hold, and $k \le n$. 

\begin{lemma}
For $\bipi \neq ([1],[2],\ldots ,[k])$ we have 
\begin{equation}\label{degreezero}
\deg(Q_\bipi(\bz);k)=0.
\end{equation}
\end{lemma}

\proof
Recall from Proposition \ref{propint} that $Q_\bipi(\bz)$ is the multidegree of a $(\CC^*)^k$-invariant cone $X_\ff$ in the tangent space of the flag manifold $\flag_k^*$ at the fixed point $\bipi$. The weights of the $(\CC^*)^k$-action on this tangent space are exactly the factors of $q_1$, namely
\[z_\tau-z_{\pi_l}\ :\ \tau \neq \pi_1,\pi_2,\ldots \pi_k; \Sigma \tau \le l, |\tau|\le 2\]
and therefore the only weight containing $z_k$ is 
\[z_{\pi_k}-z_k\]
Applying Proposition \ref{vanishlemma} with $m=k$ we arrive at \eqref{degreezero}. 
\qed 

We can thus assume that $\pi_k=[k]$, and proceed to the study of the
next residue, the one taken with respect to $z_{k-1}$. Again, assume
that $\pi_{k-1}\neq[k-1]$. As in the case of $z_k$ above,
\[\deg(q_2(\bz),k-1)=n, \deg(p_2(\bz);k-1)=0.\]
In $q_1$ the linear terms containing $z_{k-1}$ are
\begin{equation}\label{termskminusone}
z_{k-1}-z_k , z_1+z_{k-1}-z_k, z_{k-1}-z_{\pi_{k-1}}
\end{equation}

The first term here cancels with the identical term in the Vandermonde in $p_1$. The second term divides $Q_{\bipi}$, according to the following proposition from \cite{bsz} applied with $l=k-1$:
\begin{proposition}[\cite{bsz}, Proposition 7.4]\label{divisible}
Let $l\geq1$, and let $\bipi$ be an admissible sequence of
  partitions of the form $\bipi=(\pi_1,\ldots ,\pi_l,[l+1],\ldots, [k])$, where $\pi_l\neq[l]$. Then
  for $m>l$, and every partition $\tau$ such that
  $l\in\tau$, $\um\tau\leq m$, and $|\tau|>1$, we have
  \begin{equation}
    \label{zdividesq}
 (z_\tau-z_m)|Q_\bipi.
  \end{equation}
\end{proposition}
Therefore, after cancellation, all linear factors from $q_1(\bz)$ which
have nonzero coefficients in front of both $z_{k-1}$ and $z_k$ vanish, and we can apply the second option in Proposition \ref{vanishprop}, leaving us with checking the degrees of $z_k$ in the new numerator and denominator of the fraction $\frac{p'(\bz)}{q'(\bz)}$.  

Note that $\frac{Q_\bipi(\bz)}{z_1+z_{k-1}-z_k}$ is the multidegree of the same cone in a smaller vector space, namely, the cone sits in the subspace \[S=\{y_{z_1+z_{k-1}-z_k}=0\}\subset T_{p_{\bipi}}\flag_k^*,\] 
where $y_{z_1+z_{k-1}-z_k}$ is eigencoordinate corresponding to the weight $z_1+z_{k-1}-z_k$. The weights with nonzero coefficient of $z_{k-1}$ in $S$ are 
\[z_{k-1}-z_{\pi_{k-1}},z_{k-1}-z_k,\]    
and by Lemma \ref{vanishlemma} 
\[\deg(p'(\bz);k-1)\le k-2+1=k-1.\]
On the other hand $\deg(q'(\bz);k-1)=n+1$,
so we can apply the second part of Proposition \ref{vanishprop}.
In general, assume that 
\[\bipi=(\pi_1,\pi_2,\ldots ,\pi_l,[l+1],\ldots ,[k]), \pi_l\neq [l],\]
and embark on the study of the residue with respect to $z_l$. The weights containing $z_l$ in $q_1$ are 
\begin{eqnarray}
z_l-z_k,z_{l}-z_{k-1},\ldots ,z_{l}-z_{l+1} \label{vandweights} \\
z_\tau-z_s \text{ with } l \in \tau, \tau \neq l, l+1\le s\le k, \um \tau \le s    \label{restweights} \\
z_l-z_{\pi_l}
\end{eqnarray}

The weights in \eqref{vandweights} cancel out with the identical terms in $p_1(\bz)$. By Propostition \ref{divisible}, the cone, whose multidegree is $Q_\bipi(\bz)$ sits in the subspace $S$, orthogonal to the coordinates corresponding to the weights in \eqref{restweights}, and therefore $Q_{\bipi}$ is divisible by these. Using Lemma \ref{vanishlemma}, after cancellation we are left with
\[\deg(p'(\bz);l)=l-1+\deg(Q'(\bz),l) \le l-1+k-l=k-1 ; \deg(q'(\bz))=n+1,\]
again. Since $k\le n$, by applying the second option of Proposition \ref{vanishprop} we arrive at the vanishing of the residue, forcing $\pi_l$ to be $[l]$.

\section{Proof of Theorem \ref{mainthmtwo}}\label{sec:existence}

Let $X\subset \PP^{n+1}$ be a smooth projective hypersurface of degree $\deg(X)=d$. The starting point is the following proposition which tells that push-forwards of sections of the canonical line bundle on $\calx_k$ represent invariant jet differentials.  
\begin{proposition}\label{directimage}
Let $\tau$ denote the tautological rank $k$ vector bundle over the flag bundle $\flag_k(\symdotx)$ and let $\calo_{\calx_k}(1)=\wedge^k \tau |_{\calx_k}$ be the canonical line bundle on $\calx_k$. Then 
\begin{equation}\label{imageformula}
\pi_*\calo_{\calx_k}(m)\subset \calo(E_{k,m{k+1 \choose 2}}\cotx)
\end{equation}
where $\pi:\calx_k \to X$ is the projection.
\end{proposition}
\proof
By \eqref{sgamma}, the sections of the tautological bundle $\calo_{{\calx}_k}(1)$ on $\calx_k \subset \flag_k(\symdotx$ are given by Plucker coordinates on
\[\mathcal{S}^k_\g=\mathrm{Span}_\CC (v_1,v_2+v_1^2,\ldots, \sum_{j_1+\ldots +j_s=k}v_1^{j_1}\cdots v_s^{j_s}) \text{ with } v_i=\g^{(i)}/i!, \]
that is, $k \times k$ minors of the matrix $M_\g \in \Hom(\CC^k, \symdotx)$, whose $i$th row is 
\[\sum_{j_1+\ldots +j_s=i}v_1^{j_1}\cdots v_s^{j_s}\in \Sym^i T_X^*.\]
These $k\times k$ minors have weighted degree $1+2+\ldots +k={k+1 \choose 2}$ in the $\g^{(i)}$, and invariant under $U_k$.  
\qed 

Let $\calo_{\tilde{\calx}_k}(m)=\rho^* \calo_{\calx_k}$ denote the $m$-twisted canonical bundle on $\tilde{\calx}_k$, where $\rho:\tilde{\calx}_k \to \calx_k$ is the partial resolution defined in \eqref{desing}. Proposition \ref{directimage} now implies
\begin{corollary}\label{imageformula2}
For $\tilde{\pi}=\tau \circ \pi:\tilde{\calx}_k \to X$ (see diagram \eqref{diagram})
\begin{equation*} 
\tilde{\pi}^*\calo_{\tilde{\calx}_k}(m)\subseteq \calo(E_{k,m{k+1 \choose 2}}\cotx).
\end{equation*}
\end{corollary}

The following classical theorem connects global invariant jet differentials to the GGL conjecture.

\begin{theorem}[Fundamental vanishing theorem \cite{gg,dem,siu1}]\label{demailly}
Assume that there exist integers $k,m>0$ and ample line bundle
$A\to X$ such that
\[0\neq H^0(\tilde{\calx}_k,\calo_{\tilde{\calx}_{k}}(m) \otimes \pi^*A^{-1})\subset
H^0(X,E_{k,m}\cotx \otimes A^{-1})\] 
has non zero sections
$\s_1,\ldots ,\s_N$, and let $Z\subset X_k$ be the base locus of
these sections. Then every entire holomorphic curce $f:\CC \to X$
necessarily satisfies $f_{[k]}(\CC)\subset Z$. In other words, for
every global $\GG_k$-invariant differential equation $P$ vanishing on
an ample divisor, every entire holomorphic curve $f$ must satisfy
the algebraic differential equation $P(f'(t),\ldots, f^{(k)}(t))\equiv 0$.
\end{theorem}

Note, that by Theorem 1. of \cite{div2}, 
\[H^0(X,E_{k,m}\cotx \otimes A^{-1})= 0\]
holds for all $m\ge 1$ if $k<n$, so we can restrict our attention to the range $k\ge n$. On the other hand for $k>n$, the flag manifold $\flag_k(\CC^n)$ is not defined in the Snowman Model, and therefore our residue formula \eqref
{intnumberthree} does not hold. Therefore we consider the $k=n$ case only. 

To control the order of vanishing of these differential forms along
the ample divisor we choose $A$ to be --as in  \cite{dmr} -- a
proper twist of the canonical bundle of $X$. Recall that the
canonical bundle of the smooth, degree $d$ hypersurface $X$ is
\[K_X=\calo_X(d-n-2),\]
which is ample as soon as $d\ge n+3$.
The following theorem summarises the results of \S 3 in \cite{dmr}.

\begin{theorem}[Algebraic degeneracy of entire curves \cite{dmr}]\label{germtoentire}
Assume that $n=k$, and there exist a $\delta=\delta(n) >0$ and $D=D(n,\delta)$ such that
\[H^0(\tilde{\calx}_{n},\calo_{\tilde{\calx}_{n}}(m) \otimes \pi^*K_X^{-\delta m})\subseteq
H^0(X,E_{d,m}\cotx \otimes K_X^{-\delta m})\neq 0 \]
whenever $\deg(X)>D(n,\delta)$ for some $m \gg 0$. 
Then the Green-Griffiths-Lang conjecture holds whenever 
\[\deg(X) \ge \max(D(n,\delta), \frac{n^2+2n}{\delta}+n+2).\]
\end{theorem}

Following \cite{dmr} we choose $A$ to be a
proper twist of the canonical bundle of $X$, which is ample as soon as $d\ge n+3$ and we prove

\begin{theorem}\label{maintwo}
Let $X\subset \PP^{n+1}$ be a smooth complex hypersurface with ample canonical bundle, that is $\deg X\ge n+3$. If $\d=\frac{1}{n^3(n+1)}$ and $d>D(n)=2n^{10}$
then
\[\emptyset \neq H^0(\tilde{\calx}_n,\calo_{\tilde{\calx}_n}(m) \otimes \tilde{\pi}^*K_X^{-\d {n+1 \choose 2}m})\subseteq
H^0(X,E_{n,m{n+1 \choose 2}}\cotx \otimes K_X^{-\d {n+1 \choose 2}m}) \] is nonempty, provided that $\d {n+1 \choose 2}m$ is integer and Conjecture \ref{conj} holds. 
\end{theorem}
Theorem \ref{mainthmtwo} follows from Theorem \ref{germtoentire} and Theorem \ref{maintwo}. 

To prove Theorem \ref{maintwo} we use the
algebraic Morse inequalities of Demailly and Trapani to reduce the existence of global sections to the positivity of certain tautological integrals over $\tilde{\calx}_k$.  Let $L\to X$
be a holomorphic line bundle over a compact K\"ahler manifold of
dimension $n$ and $E \to X$ a holomorphic vector bundle of rank $r$.
Demailly in \cite{dem00} proved the following 
\begin{theorem}[Algebraic Morse inequalities \cite{dem00,trap}]\label{morse}
Suppose that $L=F\otimes G^{-1}$ is the difference of the nef line
bundles $F,G$. Then for any nonnegative integer $q\in \ZZ_{\ge 0}$ 
\[\sum_{j=0}^q(-1)^{q-j}h^j(X,L^{\otimes m}\otimes E)\le
r \frac{m^n}{n!}\sum_{j=0}^q(-1)^{q-j}{n \choose j}F^{n-j}\cdot
G^j+o(m^n).\] 
In particular, $q=1$ asserts that $L^{\otimes
m}\otimes E$ has a global section for $m$ large provided 
\[F^n-nF^{n-1}G>0.\]
\end{theorem}
 In order to apply this theorem we have to express $\calo_{\tilde{\calx}_{n}}(1)$ as a difference of nef bundles as follows. 

\begin{proposition}\label{nef}
Let $d\ge n+3$ and therefore $K_X$ ample. The following line bundles
are nef on $\tilde{\calx}_n$:
\begin{enumerate}
\item $\calo_{\tilde{\calx}_n}(1) \otimes \tilde{\pi}^*\calo_X(2n^2)$
\item $\tilde{\pi}^*\calo_X(2n^2)\otimes \tilde{\pi}^*K_X^{\d{n+1 \choose 2}}$ for any $\d>0$ and $\d {n+1 \choose 2}$ integer.
\end{enumerate}
\end{proposition}

\proof
Let $\calo(m)$ denote the $m$-twisted tautological bundle on $\PP^{n+1}$. Then 
$T^*_{\PP^{n+1}}\otimes \calo(2)$ is globally generated, and there is a surjective bundle map
\[(T^*_{\PP^{n+1}}\otimes \calo(2))|_X^{\otimes m} \rightarrow T^*_X\otimes \calo_X(2)^{\otimes m},\]
therefore $T^*_X\otimes \calo_X(2)$ is globally generated. 
Consequently, the left hand side of the following surjective bundle map is globally generated, 
\begin{multline}\nonumber
\wedge^n\left(T^*_X\otimes \calo_X(2)\oplus \sym^2T^*_X\otimes \calo_X(4) \oplus \ldots \oplus \sym^nT_X^* \otimes \calo_X(2n) \right)\rightarrow \\
\wedge^n\left((T^*_X \oplus \sym^2T^*_X \oplus \ldots \oplus \sym^nT_X^*)\otimes \calo_X(2n)\right)=\wedge^n\left(T^*_X \oplus \ldots \oplus \sym^nT_X^* \right)\otimes \calo_X(2n^2),
\end{multline}
and therefore the right hand side is also globally generated. So
\[\calo_{\PP(\wedge^n(\symdotxn))}(1)\otimes \pi^*\calo_X(2n^2)\]
is nef on $\calx_n$, and therefore its pull-back with $\rho$ (see diagram \eqref{diagram}) is nef too. Thus the first part of Proposition \ref{nef} is proved.
The second part follows from the standard fact that the pull-back of an ample line bundle is nef.
\qed 

Consequently, we can express $\calo_{\tilde{\calx}_n}(1)\otimes \tilde{\pi}^*K_X^{-\d{n+1 \choose 2}}$ as the following difference of two nef line bundles:
\begin{equation*}
\calo_{\tilde{\calx}_n}(1)\otimes \tilde{\pi}^*K_X^{-\d{n+1 \choose 2}}=(\calo_{\tilde{\calx}_n}(1) \otimes
\tilde{\pi}^*\calo_X(2n^2))\otimes (\tilde{\pi}^*\calo_X(2n^2)\otimes \tilde{\pi}^*K_X^{\d{n+1 \choose 2}})^{-1}.
\end{equation*}
Theorem \ref{maintwo} follows from the Morse inequalities by proving that the following top form on $\tilde{\calx}_n$ is positive if $\d=\frac{1}{n^3(n+1)}$ and $d>D(n)=2n^{10}$:
\begin{multline}\label{intnumbersnowman}
I(n,\d,d)=\\
=c_1(\calo_{\tilde{\calx}_n}(1) \otimes \tilde{\pi}^*\calo_X(2n^2))^{n^2}-
n^2c_1(\calo_{\tilde{\calx}_n}(1) \otimes
\tilde{\pi}^*\calo_X(2n^2)^{(n^2-1)}c_1(\tilde{\pi}^*\calo_X(2n^2)\otimes
\tilde{\pi}^*K_X^{\d{n+1 \choose 2}}).
\end{multline}

Recall the notations $h=c_1(\calo_X(1)), u=c_1(\calo_{\tilde{\calx}_n}(1))$, and $c_1=c_1(T_X)$ for the corresponding first Chern classes. Then $c_1(K_X)=-c_1=(d-n-2)h$, and by dropping $\tilde{\pi}^*$ from our formula we define the following polynomial in $z_1,\ldots, z_n, h$:
\begin{multline}\label{rform}
I_{n,\d,d}(\bz,h)=(z_1+\ldots +z_n+2n^2h)^{n^2}-\\
-n^2(z_1+\ldots +z_n+2n^2h)^{n^2-1}(2n^2h+\d {n+1 \choose 2}(d-n-2)h).
\end{multline}
Integration over $X$ on both sides of \eqref{intnumberthree}  then gives
\begin{theorem}\label{finalformulaone} Let $I(n,\d,d)$ be the intersection number defined in \eqref{intnumbersnowman} on the Snowman Model $\tilde{\calx}_n$. Then 
\begin{equation*}
I(n,\d,d)=\int_X \sires \frac{
Q_n(\bz)\,\prod_{m<l}(z_m-z_l) I_{n,\d,d}(\bz,h)}{
\prod_{m+r\le l \le n}
(z_m+z_r-z_l)\prod_{l=1}^n\prod_{i=1}^n(\lambda_i-z_l)} \,\dbz
\end{equation*} 
where integration on the right hand side means the substitution $h^n=d$. 
\end{theorem}
Here $\lambda_1,\ldots, \lambda_n$ represent the Chern roots of $\cotx$, and therefore $-\l_1,\ldots ,-\l_n$ are the Chern roots of $T_X$. Since $X\subset \PP^{n+1}$ is a projective htpersurface we can eliminate these using that the Chern classes of $X$ are expressible with $d=\deg(X),h$ via
\[(1+h)^{n+2}=(1+dh)c(X),\]
where $c(X)=c(T_X)$ is the total Chern class of $X$. This gives 
\begin{multline}
\frac{1}{\prod_{l=1}^n\prod_{i=1}^n(\lambda_i-z_l)}=\frac{(-1)^n}{(z_1\ldots z_n)^n}\frac{1}{\prod_{i,l=1}^n(1-\frac{\l_i}{z_l})}=\frac{(-1)^n}{(z_1\ldots z_n)^n}\frac{1}{\prod_{l=1}^n c(1/z_l)}=\\
=\frac{(-1)^n}{(z_1\ldots z_n)^n}\prod_{l=1}^n\frac{1+\frac{dh}{z_l}}{(1+\frac{h}{z_l})^{n+2}}=\frac{(-1)^n}{(z_1\ldots z_n)^n}\prod_{l=1}^n\left(1+\frac{dh}{z_l}\right)\prod_{l=1}^n\left(1-\frac{h}{z_l}+\frac{h^2}{z_l^2}-\ldots \right)^{n+2}
\end{multline}
and therefore Theorem \ref{finalformulaone} can be rewritten as 
\begin{multline}\label{finalformulatwo}
I(n,\d,d)=\\
\int_X \sires \underbrace{\frac{(-1)^n
Q_n(\bz)\prod_{m<l}(z_m-z_l)I_{n,\d,d}(\bz,h)\dbz}{
\prod_{m+r\le l \le n}
(z_m+z_r-z_l) (z_1\ldots z_n)^n }}_{A^1} \underbrace{\prod_{l=1}^n\left(1+\frac{dh}{z_l}\right)\prod_{l=1}^n\left(1-\frac{h}{z_l}+\ldots \right)^{n+2}}_{A^2}
\end{multline}
where from  \eqref{rform} 
\begin{multline}\nonumber
I_{n,\d,d}(\bz,h)=(z_1+\ldots +z_n+2n^2h)^{n^2-1}\cdot \\ \cdot \left(z_1+\ldots +z_n-\d n^2{n+1 \choose 2}dh-\left(2n^4-n^2\d (n+2){n+1 \choose 2}-2n^2\right)h\right).
\end{multline}

The residue is by definition the coefficient of $\frac{1}{z_1\ldots z_n}$ in the appropriate Laurent expansion of the big rational expression in $z_1,\ldots, z_n,n, d, h$ and $\d$, multiplied by $(-1)^n$. We can therefore omit the $(-1)^n$ factor from the numerator and simply compute the corresponding coefficient. The result is a polynomial in $n,d,h,\d$, and in fact, a relatively easy argument shows that it is a polynomial in $n,d,\d$ multiplied by $h^n$ 

Indeed, giving degree $1$ to $z_1,\ldots ,z_n,h$ and $0$ to $n,d,\d$, the rational expression in the residue has total degree $0$. Therefore the coefficient of $\frac{1}{z_1\ldots z_n}$ has degree $n$, so it has the form $h^n p(n,d,\d)$ with a polynomial $p$. Since $\int_X h^n=d$, integration over $X$ is simply a substitution $h^n=d$, resulting in the equation $I(n,\d,d)=dp(n,\d,d)$ for some polynomial $p(n,\delta,d)$.

\subsection{A first look at the iterated residue formula}\label{subsec:residue}

To overcome the difficulties in handling the rational expression in \eqref{finalformulatwo}, we introduce the following notations.
 
\noindent \textbf{Notation} For $\bi=(i_1,\ldots i_n) \in \ZZ^n$ let
\begin{equation}\label{theta}
\rho_{\bi}=\coeff_{\bz^{\bi}}\frac{
 Q_n(\bz)\,\prod_{m<l}(z_m-z_l) (z_1+\ldots +z_n)^{n^2+i_1+\ldots +i_n}}{
\prod_{m+r\le l \le n}
(z_m+z_r-z_l) (z_1\ldots z_n)^n }
\end{equation}
stand for the coefficient of $\bz^{\bi}$. The total degree of the rational expression in \eqref{theta} is $\Sigma\bi=i_1+\ldots +i_n$ and therefore the coefficient of $z_1^{i_1}\ldots z_n^{i_n}$ might be nonzero. 

\begin{proposition}\label{leadingcoeff}
\begin{enumerate}
\item
$I(n,\d,d)$ is a polynomial in $d$ of degree $n+1$ without constant term:
\[I(n,\d,d)=p_{n+1}(n,\d)d^{n+1}+p_{n}(n,\d)d^{n}+\ldots +p_1(n,\d)d\]
where $p_i(n,\d)$ is linear in $\d$ and polynomial in $n$ for all $i$.
\item 
The leading coefficient is $p_{n+1}(n,\d)=\left(1-n^2{n+1 \choose 2}\delta\right)\rho_{\mathbf{0}}$.
\end{enumerate}
\end{proposition}
\proof
The first part follows from the previous remarks. 
To prove the second part we study the formula \eqref{finalformulatwo}. To get $d^{n+1}$ we either have to choose all the $\frac{dh}{z_l}$ terms in the product $\prod_{l=1}^n\left(1+\frac{dh}{z_l} \right)$, or we need to pick the $\frac{dh}{z_s}$ term in $A^1$ and pair up with the terms $\frac{dh}{z_l}$, $l\neq s$ in the product $\prod_{l=1}^n\left(1+\frac{dh}{z_l} \right)$. This gives us
\begin{equation}\nonumber
p_{n+1}(n,\d)=\rho_{\mathbf{0}}-n^2{n+1 \choose 2}\delta\sum_{s=1}^n \rho_{-e_s} 
\end{equation}
where $e_s=(0,\ldots, 1, \ldots, 0)$ stands for the $s$th basis vector.
By definiton 
\begin{multline}\nonumber 
\sum_{s=1}^n \rho_{-e_s}=\sum_{s=1}^n \coeff_{z_s^{-1}}\frac{
 (-1)^nQ_n(\bz)\,\prod_{m<l}(z_m-z_l) (z_1+\ldots +z_n)^{n^2-1}}{
\prod_{m+r\le l \le n}
(z_m+z_r-z_l) (z_1\ldots z_n)^n }=\\
=\sum_{s=1}^n \coeff_{\bz^{\mathbf{0}}}\frac{
(-1)^n Q_n(\bz)\,\prod_{m<l}(z_m-z_l) z_s(z_1+\ldots +z_n)^{n^2-1}}{
\prod_{m+r\le l \le n}
(z_m+z_r-z_l) (z_1\ldots z_n)^n }=\rho_{\mathbf{0}}.
\end{multline}
\qed 

A necessary condition for $I(n,\d,d)>0$ to be positive for any $d\gg 0$ is a positive leading coefficient, that is, $(1-n^2{n+1 \choose 2}\delta)\rho_{\mathbf{0}}>0$. Using the Thom polynomial generating function \eqref{tpgenerating} introduced in the Introduction and the definition of $\rho_{\mathbf{0}}$ we can write
\begin{equation}\label{rho0}
\rho_{\mathbf{0}}=\sum_{\bi\in (\ZZ^{\ge 0})^n,\Sigma \bi=n^2}\Tp_{(n,\ldots ,n)-\bi}{n^2 \choose i_1\  i_2\  \ldots i_n}.
\end{equation}
By Remark \ref{remarkq} and the positivity property of the multidegree (see Proposition \ref{epdprops})  
\[Q_n(\bz)=P(z_m+z_r-z_l:m+r\le l\le n)\]
is a polynomial in the weight variables $z_m+z_r-z_l$ with positive integer coefficient. More precisely, we saw in Sect.~\ref{subsec:epdmult} (see \eqref{mdegformula}) that the monomials of $P$ are equivariant duals of the irreducible components of a flat deformation of the cone $N_k$ of Remark \ref{remarkq}. These irreducible components are subspaces with positive multiplicities (see \eqref{mdegformula}) and $\Tp_{\mathbf{0}}$ is the sum of these multiplicities (see \cite{thesis}) therefore it is a positive integer. That is, at least one term in \eqref{rho0} is positive, and the first part of Conjecture \ref{conj} implies that $\rho_{\mathbf{0}}>0$. 
A straightforward corollary is  
\begin{proposition}\label{posleading}
For $\delta < \frac{2}{n^3(n+1)}$ the leading coefficient $p_{n+1}(n,\delta)>0$ is positive, and therefore $I_X(d,n,\d)>0$ for $d\gg 0$. 
\end{proposition}

According to Proposition \ref{leadingcoeff} (2), we cannot expect a better than polynomial bound for the Green-Griffiths-Lang conjecture from this model. We fix $\delta$ to be $\delta=\frac{1}{n^3(n+1)}$ for the rest of the paper.  

\subsection{Estimation of the coefficients}\label{subsec:coeffs} According to Proposition \ref{leadingcoeff} we have to prove the positivity of the polynomial $I(n,\d,d)=p_{n+1}(n,\d)d^{n+1}+p_{n}(n,\d)d^{n}+\ldots +p_1(n,\d)d$ whose leading coefficient is positive due to Proposition \ref{posleading}. We prove this by showing that 
\[|p_{n+1-l}|<n^{10l}p_{n+1}\]
holds for $1\le l \le n+1$ and then by using the following elementary statement:
\begin{lemma}[Fujiwara bound]\label{estimation}
If $p(d)=p_nd^n+p_{n-1}d^{n-1}+\ldots +p_1d+p_0\in \RR[d]$ satisfies the inequalities 
\[p_n>0;\ \ |p_{n-l}|<D^l |p_n| \text{ for } l=1,\ldots n,\]
then $p(d)>0$ for $d>2D$.  
\end{lemma}

Let $\delta=\frac{1}{n^3(n+1)}$ be fixed. If the degree of $z_1,\ldots , z_n$ and $h$ is $1$, then the denominator and numerator of $A^1$ are homogeneous polynomials of the same degree, and therefore in the Laurent expansion we have terms
\[\frac{(dh)^{\varepsilon}h^{m}\bz^{\mathbf{a}}}{\bz^{\mathbf{b}}} \text{ with } \baa,\bbb\in \ZZ_{\ge 0}^n,\ \varepsilon=0 \text{ or } 1,  \varepsilon+m+\Sigma \mathbf{a}=\Sigma \mathbf{b},\ a_ib_i=0 \text{ for } 1\le i \le n.\] 
Let $A^1_{(dh)^\varepsilon h^m \bz^{\mathbf{a}-\mathbf{b}}}$ denote the coefficient of this term in $A^1$. With $\delta=\frac{1}{n^3(n+1)}$ we have
\begin{equation}\label{etatheta}
A^1_{(dh)^1 h^{\Sigma \mathbf{b}-\Sigma \mathbf{a}-1} \bz^{\mathbf{a}-\mathbf{b}}}=-\rho_{\mathbf{a}-\mathbf{b}}\frac{1}{2}{n^2-1 \choose \Sigma \mathbf{b}-\Sigma \mathbf{a}-1}(2n^2)^{\Sigma \mathbf{b}-\Sigma \mathbf{a}-1}
\end{equation}
and 
\begin{multline}\nonumber
A^1_{(dh)^0 h^{\Sigma \mathbf{b}-\Sigma \mathbf{a}} \bz^{\mathbf{a}-\mathbf{b}}}=\\
\rho_{\mathbf{a}-\mathbf{b}}
\left({n^2-1 \choose \Sigma \mathbf{b}-\Sigma \mathbf{a}}(2n^2)^{\Sigma \mathbf{b}-\Sigma \mathbf{a}}-(2n^4-2n^2-\frac{1}{2}(n+2)){n^2-1 \choose \Sigma \mathbf{b}-\Sigma \mathbf{a}-1}(2n^2)^{\Sigma \mathbf{b}-\Sigma \mathbf{a}-1} \right).
\end{multline}
These give us $\left|A^1_{(dh)^0 h^{\Sigma \mathbf{b}-\Sigma \mathbf{a}-1} \bz^{\mathbf{a}-\mathbf{b}}}\right|<2n^4\left|A^1_{(dh)^1 h^{\Sigma \mathbf{b}-\Sigma \mathbf{a}-1} \bz^{\mathbf{a}-\mathbf{b}}}\right|$ and therefore since 
\[A^1_{h^{\Sigma \mathbf{b}-\Sigma \mathbf{a}} \bz^{\mathbf{a}-\mathbf{b}}}=dA^1_{(dh)^1 h^{\Sigma \mathbf{b}-\Sigma \mathbf{a}-1} \bz^{\mathbf{a}-\mathbf{b}}}+A^1_{(dh)^0 h^{\Sigma \mathbf{b}-\Sigma \mathbf{a}} \bz^{\mathbf{a}-\mathbf{b}}},\] we arrive at 
\begin{lemma}\label{etalemma}
For $\delta=\frac{1}{n^3(n+1)}$ and $d>2n^5$ we have 
$\left|A^1_{(dh)^0 h^{\Sigma \mathbf{b}-\Sigma \mathbf{a}} \bz^{\mathbf{a}-\mathbf{b}}}\right|<\frac{1}{n}\left|A^1_{h^{\Sigma \mathbf{b}-\Sigma \mathbf{a}} \bz^{\mathbf{a}-\mathbf{b}}}\right|$
and therefore for $\Sigma \mathbf{b}-\Sigma \mathbf{a}\ge 1$
\begin{equation}\label{etaineqtwo}
A^1_{h^{\Sigma \mathbf{b}-\Sigma \mathbf{a}} \bz^{\mathbf{a}-\mathbf{b}}}=dC_{\baa,\bbb}A^1_{(dh)^1 h^{\Sigma \mathbf{b}-\Sigma \mathbf{a}-1} \bz^{\mathbf{a}-\mathbf{b}}} \text{ for some } 1-\frac{1}{n}<|C_{\baa,\bbb}| <1+\frac{1}{n}.
\end{equation}
\end{lemma}
Let $\mathbf{1}=(1,\ldots,1)$. According to \eqref{finalformulatwo} and Lemma \ref{etalemma} for $d>2n^5$ we have
\begin{multline}\label{infix}
I(n,\delta,d)=d\sum_{\bbb \in \{0,1\}^n,\baa}A^1_{h^{\Sigma \mathbf{b}-\Sigma \mathbf{a}} \bz^{\mathbf{a}-\mathbf{b}}}A^2_{h^{n-\Sigma \mathbf{b}+\Sigma \mathbf{a}}\bz^{\mathbf{b}-\mathbf{a}-\mathbf{1}}}=\\
d^2\sum_{\substack{\bbb \in \{0,1\}^n \\ \Sigma \baa<\Sigma \bbb}}C_{\baa,\bbb}A^1_{(dh)^1h^{\Sigma \mathbf{b}-\Sigma \mathbf{a}-1} \bz^{\mathbf{a}-\mathbf{b}}}A^2_{h^{n-\Sigma \mathbf{b}+\Sigma \mathbf{a}}\bz^{\mathbf{b}-\mathbf{a}-\mathbf{1}}}+d\sum_{\substack{\bbb \in \{0,1\}^n\\ \Sigma \baa=\Sigma \bbb}} A^1_{(dh)^0h^0 \bz^{\mathbf{a}-\mathbf{b}}}A^2_{h^{n}\bz^{\mathbf{b}-\mathbf{a}-\mathbf{1}}}.
\end{multline}
Indeed, the left hand side is by definition the coefficient of $\frac{1}{z_1\ldots z_n}$ following the substitution $h^n=d$. Now $\bbb \in \{0,1\}^n$ means that $z^{\bbb}$ is square-free, which is necessary to get $z_1\ldots z_n$ in the denominator. Note that by definition $A^1_{(dh)^0h^0 \bz^{\mathbf{a}-\mathbf{b}}}=\rho_{\baa-\bbb}$.
Introduce the notation $B(\bz,h)=\prod_{l=1}^n\left(1-\frac{h}{z_l}+\frac{h^2}{z_l^2}-\ldots \right)^{n+2}$
for the second term of $A^2$. For simplicity, we denote by $B_{\bz^\bi}$ the coefficient of $h^{-\Sigma\bi}\bz^\bi$ in $B$. Since $\bbb \subset \{0,1\}^n$ it follows that $\mathbf{1}-\bbb \in \{0,1\}^n$) and the condition $a_ib_i$ ($1\le i \le n$) is equivalent to saying that $a_i >0 \Rightarrow i \in \mathbf{1}-\bbb$. Therefore
\begin{equation}\label{A2}
A^2_{h^{n-\Sigma \mathbf{b}+\Sigma \mathbf{a}}\bz^{\mathbf{b}-\mathbf{a}-\mathbf{1}}}=d^{n-\Sigma \bbb}B_{\bz^{-\baa}}+d^{n-1-\Sigma \bbb}\sum_{s_1 \in \mathbf{1}-\bbb}B_{\bz^{-\baa-e_{s_1}}}+d^{n-2-\Sigma \bbb}\sum_{s_1,s_2 \in \mathbf{1}- \bbb}B_{\bz^{-\baa-e_{s_1}-e_{s_2}}}+ \ldots.
\end{equation}
Substituting \eqref{etatheta} and \eqref{A2} into \eqref{infix} we arrive at the following formula:
\begin{multline}\nonumber
I(n,\delta,d)=-d^2\sum_{\substack{\bbb \in \{0,1\}^n\\ \Sigma \baa<\Sigma \bbb}}C_{\baa,\bbb}\rho_{\mathbf{a}-\mathbf{b}}\frac{1}{2}{n^2-1 \choose \Sigma \bbb-\Sigma \baa-1}(2n^2)^{\Sigma \bbb-\Sigma \baa-1}\cdot \\
\left(d^{n-\Sigma \bbb}B_{\bz^{-\baa}}+d^{n-1-\Sigma \bbb}\sum_{s \in \mathbf{1}-\bbb}B_{\bz^{-\baa-e_s}}+d^{n-2-\Sigma \bbb}\sum_{s_1,s_2 \in \mathbf{1}- \bbb}B_{\bz^{-\baa-e_{s_1}-e_{s_2}}}+ \ldots \right)\\
+d\sum_{\substack{\bbb \in \{0,1\}^n \\ \Sigma \baa=\Sigma \bbb}} \rho_{\baa-\bbb}\left(d^{n-\Sigma \bbb}B_{\bz^{-\baa}}+d^{n-1-\Sigma \bbb}\sum_{s \in \mathbf{1}-\bbb}B_{\bz^{-\baa-e_s}}+d^{n-2-\Sigma \bbb}\sum_{s_1,s_2 \in \mathbf{1}- \bbb}B_{\bz^{-\baa-e_{s_1}-e_{s_2}}}+ \ldots.\right).
\end{multline}
After rearranging this expression as a polynomial of $d$, the coefficient of $d^{n+1-l}$ is
\begin{multline}\label{coefficient}
p_{n+1-l}=\sum_{r=0}^l \sum_{\substack{\bbb \in \{0,1\}^n \\ \Sigma \baa=\Sigma \bbb=r, \baa \bbb=0}} \sum_{\substack{\mathbf{s}\subset \mathbf{1}-\bbb \\ \Sigma \mathbf{s}=l-r}}\rho_{\baa-\bbb}B_{\bz^{-\baa-\mathbf{s}}}- \\
-\sum_{r=1}^{l+1} \sum_{m=1}^r \sum_{\substack{\bbb \in \{0,1\}^n \\ \Sigma \baa=r-m \Sigma \bbb=r, \\ \baa \bbb=0}} \frac{1}{2}{n^2-1 \choose m-1}(2n^2)^{m-1}\sum_{\substack{\mathbf{s}\subset \mathbf{1}-\bbb \\ \Sigma \mathbf{s}=l-r}}C_{\baa,\bbb}\rho_{\baa-\bbb}B_{\bz^{-\baa-\mathbf{s}}}.  
\end{multline}
\begin{lemma}\label{mainlemma} Conjecture \ref{conj} implies $\sum_{\substack{\bbb \in \{0,1\}^n \\ \Sigma \baa=r-m, \Sigma \bbb=r \\ \baa \bbb=0}}\rho_{\baa-\bbb}<n^{8r-7m}\rho_{\mathbf{0}}$ for $0\le m\le r \le n$.
\end{lemma}
\proof
By definition we have
\begin{equation}\label{tocontinue}
\sum_{\substack{\bbb \in \{0,1\}^n \\ \Sigma \baa=r-m \Sigma \bbb=r, \\ \baa \bbb=0}}\rho_{\baa-\bbb}=\coeff_{\bz^{\mathbf{0}}}\frac{
 Q_n(\bz)\,\prod_{m<l}(z_m-z_l) (z_1+\ldots +z_n)^{n^2-m}}{
\prod_{m+r\le l \le n}
(z_m+z_r-z_l) (z_1\ldots z_n)^n }\cdot \sum_{\substack{\bbb \in \{0,1\}^n \\ \sum \baa=r-m \Sigma \bbb=r, \\ \baa \bbb=0}}\bz^{\bbb-\baa}.
\end{equation}
For $i_1+\ldots +i_n=n^2$ we have
\begin{equation}\nonumber
\coeff_{\bz^\bi}((z_1+\ldots +z_n)^{n^2-m}\cdot \sum_{\substack{\bbb \in \{0,1\}^n \\ \Sigma \baa=r-m, \Sigma \bbb=r \\ \baa \bbb=0}}\bz^{\bbb-\baa})=\sum_{\substack{\bbb \in \{0,1\}^n \\ \Sigma \baa=r-m, \Sigma \bbb=r \\ \baa \bbb=0, \mathbf{i}+\baa-\bbb \ge 0}}\frac{(n^2-m)!}{\prod_{t=1}^n(i_t+a_t-b_t)!}<
\end{equation}
\begin{equation}\\ \nonumber
<\sum_{\substack{\bbb \in \{0,1\}^n \\ \Sigma \baa=r-m, \Sigma \bbb=r \\ \baa \bbb=0, \mathbf{i}+\baa-\bbb \ge 0}}\sum_{\substack{\bbb' \subset \bbb \\ \Sigma \bbb'=r-m}}\frac{(n^2)!}{\prod_{a_s>0}(i_s+a_s)!\prod_{s\in \bbb'}(i_s-1)!\prod_{s\in [n]\backslash (\baa \cup \bbb')}i_s!}<
\end{equation}
\begin{equation}
<n^m\sum_{\substack{\bbb \in \{0,1\}^n \\ \Sigma \baa=r-m, \Sigma \bbb=r-m \\ \baa \bbb=0, \mathbf{i}+\baa-\bbb \ge 0}}\coeff_{\bz^{\mathbf{i}+\baa-\bbb}}(z_1+\ldots +z_n)^{n^2}.
\end{equation}
Recall from Sect.~\ref{sec:intro} the notation 
\[\mathrm{Tp}_n(\bz)=\frac{
 Q_n(\bz)\,\prod_{m<l\le n}(z_m-z_l) }{
\prod_{m+r\le l \le n}
(z_m+z_r-z_l) }\]
for the Thom series and $\mathrm{Tp}_{\bi}$ for the coefficient of $\bz^{\bi}$ in $\mathrm{Tp}_n$.
For the sake of convenience in he rest of the computation we use the shorthand notation $\coeff_{\bi}$ for $\coeff_{\bz^{\bi}}$ and we do not display the conditions $\bbb \in \{0,1\}^n$ and $\baa \bbb=0$ under the summation signs. By \eqref{tocontinue} and the first part of Conjecture \ref{conj} we get
\begin{multline}\label{long}
|\sum_{\Sigma \baa=r-m, \Sigma \bbb=r}\rho_{\baa-\bbb}|=|\sum_{\Sigma \bi=0}\Tp_{-\bi} \cdot \coeff_{\bi+n\cdot \mathbf{1}}((z_1+\ldots +z_n)^{n^2-m} \cdot \sum_{\Sigma \baa=r-m, \Sigma \bbb=r}\bz^{\bbb-\baa})|<\\
<n^m \sum_{\Sigma \bi=0}
 \Tp_{-\bi} \cdot  \sum_{\substack{\Sigma \baa=\Sigma \bbb=r-m, \\ \mathbf{i}+n\cdot \mathbf{1}+\baa-\bbb \ge 0}}\coeff_{\bi+n\cdot \mathbf{1}+\baa-\bbb}(z_1+\ldots +z_n)^{n^2}=\\
=n^m\sum_{\substack{\Sigma \mathbf{i}=0 \\ \mathbf{i}+n \cdot \mathbf{1} \ge \mathbf{0}}}
\sum_{\substack{\Sigma \baa=\Sigma \bbb=r-m \\ \mathbf{i}+n\cdot \mathbf{1}+ \baa-\bbb \ge \mathbf{0}}} \Tp_{-\bi}\cdot  \coeff_{\bi+n \cdot \mathbf{1}}(z_1+\ldots +z_n)^{n^2}\cdot \frac{\coeff_{\mathbf{i}+n\cdot \mathbf{1}+\baa-\bbb}(z_1+\ldots +z_n)^{n^2}}{\coeff_{\mathbf{i}+n \cdot \mathbf{1}}(z_1+\ldots +z_n)^{n^2}}+\\
+n^m \sum_{\substack{\Sigma \mathbf{i}=0 \\ \exists s,i_s<-n}}
\sum_{\substack{\Sigma \baa=\Sigma \bbb=r-m \\ \mathbf{i}+n \cdot \mathbf{1}+\baa-\bbb \ge \mathbf{0}}} \Tp_{-\bi}\cdot \coeff_{\bi+n \cdot \mathbf{1}+\baa-\bbb}(z_1+\ldots +z_n)^{n^2}.
\end{multline}
To estimate the first sum notice that 
\begin{equation}\label{multinomal}
\frac{\coeff_{\bi+n \cdot \mathbf{1}+ \baa-\bbb}(z_1+\ldots +z_n)^{n^2}}{\coeff_{\bi+n\cdot \mathbf{1}}(z_1+\ldots +z_n)^{n^2}}<n^{2\Sigma \bbb}
\end{equation}
and for given $\bi$ the number of pairs $(\baa, \bbb)$ with the conditions is not more than this number after dropping the positivity condition $\mathbf{i}+n\cdot \mathbf{1}+\baa-\bbb \ge \mathbf{0}$:
\begin{multline} \label{upestimate}
\sharp\{(\baa,\bbb):\bbb \in \{0,1\}^n  \Sigma \baa=\Sigma \bbb=r-m,  \baa \bbb=0\}\le {n \choose \Sigma \bbb}\cdot {n-\Sigma \bbb+\Sigma \baa \choose \Sigma \baa}<n^{2\Sigma \bbb}.
\end{multline}
For the second sum in \eqref{long} we do something similar: for any $\bi$ in the sum with $\Sigma \bi=0$ and at least one coordinate less than $-n$ we find a sequence of vectors $-\bi=\bj_0,\bj_1,\ldots \bj_{r-m}=-\tilde{\bi}$ such that i) $\bj_{s+1}$ is a predecessor of $\bj_s$ such that $\Tp_{\bj_{s+1}}\neq 0$ for $0\le s \le r-m-1$ and ii) $\tilde{\bi}+n \cdot \mathbf{1} \ge \mathbf{0}$. This sequence exists because $\mathbf{i}+n\cdot \mathbf{1}+\baa-\bbb \ge \mathbf{0}$ for some $\baa,\bbb$ with $\Sigma \baa=\Sigma \bbb=r-m$ and by taking repeatedly the predecessors we remove the highest coordinates of $-\bi$ until all coordinates fall below $n$. We can rewrite the corresponding term in \eqref{long} as 
\[\frac{\Tp_{-\bi}}{\Tp_{-\tilde{\bi}}}\cdot \frac{\coeff_{\mathbf{i}+n\cdot \mathbf{1}+\baa-\bbb}(z_1+\ldots +z_n)^{n^2}}{\coeff_{\tilde{\bi}+n \cdot \mathbf{1}}(z_1+\ldots +z_n)^{n^2}}\cdot \Tp_{-\tilde{\bi}}\cdot  \coeff_{\tilde{\bi}+n \cdot \mathbf{1}}(z_1+\ldots +z_n)^{n^2}.\]
By the second part of Conjecture \ref{conj} $\Tp_{\tilde{\bi}} \neq 0$ and $\frac{\Tp_{-\bi}}{\Tp_{-\tilde{\bi}}}<n^{2\Sigma\bbb}$. Moreover, since $\Sigma(\bi+\baa-\bbb-\tilde{\bi})\le 4(r-m)$, we have similarly to \eqref{multinomal} 
\[\frac{\coeff_{\bi+n \cdot \mathbf{1}+ \baa-\bbb}(z_1+\ldots +z_n)^{n^2}}{\coeff_{\tilde{\bi}+n\cdot \mathbf{1}}(z_1+\ldots +z_n)^{n^2}}<n^{4(r-m)}.\]
Finally, similarly to \eqref{upestimate}, the number pairs $(\baa,\bbb)$ in the second sum of \eqref{long} for any $\bi$ is not more than $n^{2(r-m)}$. Putting these together we get
\[\sum_{\substack{\bbb \in \{0,1\}^n \\ \Sigma \baa=r-m, \Sigma \bbb=r \\ \baa \bbb=0}}\rho_{\baa-\bbb}<n^{8r-7m}\sum_{\Sigma \mathbf{i}=0}\Tp_{\bz^{-\mathbf{i}}}\coeff_{\bi+n\cdot \mathbf{1}}(z_1+\ldots +z_n)^{n^2}=n^{8r-7m}\rho_{\mathbf{0}},\]
and Lemma \ref{mainlemma} is proved. 
\qed 

An easy computation shows that for $\bi=(i_1,\ldots, i_n) \in \ZZ_{\ge 0}^n$ 
$B_{\bz^\bi}= (-1)^{\Sigma \bi} \prod_{s=1}^n {n+i_s+1 \choose i_s}$. This implies that for $\bi \in (\mathbb{Z}^{\ge 0})^n$ 
$|B_{\bz^\bi}|\le (n+2)^{\Sigma \bi}$.
Substitute this and Lemma \ref{mainlemma} into the expression \eqref{coefficient} for the coefficient of $d^{n+1-l}$. 
The first term in \eqref{coefficient} can be estimated as
\[\sum_{r=0}^l \sum_{\substack{\bbb \in \{0,1\}^n \\ \Sigma \baa=\Sigma \bbb=r, \baa \bbb=0}}\sum_{\substack{\mathbf{s}\subset \mathbf{1}-\bbb \\ \Sigma \mathbf{s}=l-r}}\rho_{\baa-\bbb}B_{\bz^{-\baa-\mathbf{s}}}<
\sum_{r=0}^l {n-r \choose n-l} (n+2)^l n^{8r}\rho_{\mathbf{0}}<\frac{1}{4}n^{10l}\rho_{\mathbf{0}},\]
and the second term as
\begin{multline}\nonumber
\sum_{r=1}^{l+1} \sum_{m=1}^r \sum_{\substack{\bbb \in \{0,1\}^n \\ \Sigma \baa=r-m \Sigma \bbb=r, \\ \baa \bbb=0}} \frac{1}{2}{n^2-1 \choose m-1}(2n^2)^{m-1}\sum_{\substack{\mathbf{s}\subset \mathbf{1}-\bbb \\ \Sigma \mathbf{s}=l-r}}C_{\baa,\bbb}\rho_{\baa-\bbb}B_{\bz^{-\baa-\mathbf{s}}}<\\ \nonumber
<\sum_{r=1}^{l+1} \sum_{m=1}^r \frac{1}{2}{n^2-1 \choose m-1}(2n^2)^{m-1} {n-r \choose n-l} (n+2)^{l-m}n^{8r-7m}\rho_{\mathbf{0}}<\frac{1}{4}n^{10l}\rho_{\mathbf{0}}.
\end{multline}
Recall that with $\delta=\frac{1}{n^2(n+1)}$ the leading coefficient is $p_{n+1}=\frac{1}{2}\rho_{\mathbf{0}}$. Therefore adding these two in \eqref{coefficient} we arrive at
\[|p_{n+1-l}|<n^{10l}p_{n+1}\]
Using Lemma \ref{estimation} this proves Theorem \ref{maintwo}, which together with Theorem \ref{germtoentire} gives Theorem \ref{mainthmtwo}.

\section{On Conjecture \ref{conj}}\label{sec:conjecture}

In this last section we motivate Conjecture \ref{conj} with some observations.
The first part of Conjecture \ref{conj} is the special case of the more general conjecture of Rim\'anyi \cite{rimanyi} saying that the Thom series of any contact singularity class is positive. Morin singularities are probably the most studied contact classes, and further examples are computed in \cite{rf}. Note that Pragacz and Weber in \cite{pragaczweber} prove Shur positvity of Thom polynomials for contact classes, but we need positivity in Chern classes for our argument to work.  

The second part of Conjecture \ref{conj} is based on the observation that in the known examples for Thom polynomials there are no isolated nonzero coefficients in the sense that for any $\bi$ with $\Sigma \bi=0$ and $\Tp_{\bi}>0$ there is a chain $\mathbf{0}=\bi_0,\bi_1,\ldots ,\bi_r=\bi$ of vectors such that $\Tp_{\bi_j}>0$ for $0\le j\le r$ and $\bi_{j}$ is a predecessor of $\bi_{j+1}$ for $0\le j \le r-1$. In short, any nonzero coefficient can be approached by elementary steps from the leading coefficient $\Tp_{\mathbf{0}}$ moving only on positive coefficients. 
\subsection{The convergence of $\Tp_k$}\label{subsec:conv}
The Laurent expansion of $\Tp_k(z_1,\ldots, z_k)$ is convergent when $z_i+z_j<z_l$ holds for $i+j \le l \le k$, as the power series 
\[\frac{1}{z_i+z_j-z_l}=\frac{-1}{z_l}\left(1+\frac{z_i+z_j}{z_l}+
\frac{(z_i+z_j)^2}{z_l^2}+\ldots \right)\]
are convergent. If the positivity conjecture Conjecture \ref{conj} (1) is true, then any subseries is convergent, in particular for any $\bi=(i_1,\ldots, i_k),\Sigma \bi=0$ and $1\le l,m \le k$ the series
\[\sum_{s=0}^\infty \Tp_{\bi+s(e_l-e_m)}\bz^{\bi+s(e_l-e_m)}\]
is convergent with the substitution $z_j=j^2$. That is,
\[\Tp_{\bi}\cdot 1^{2i_1}\cdot 2^{2i_2}\ldots \cdot k^{2i_k}\sum_{s=0}^\infty \frac{\Tp_{\bi+s(e_l-e_m)}}{\Tp_{\bi}}\left(\frac{l}{m}
\right)^{2s}<\infty\]
But $\frac{l}{m}\ge \frac{1}{n}$, so $\sum_{s=0}^\infty \frac{\Tp_{\bi+s(e_l-e_m)}}{\Tp_{\bi}\cdot n^{2s}}<\infty$, suggesting the inequality in the second part of Conjecture \ref{conj}. 
\subsection{Checking the known cases $n=m$, $k \le 8$}\label{subsec:known}	
Table \ref{knownthom} lists the known Thom polynomials $\Tp_k^{0}(c_1,c_2,\ldots)$ form \cite{rimanyi} (Theorem 5.1).
\begin{table}
\caption{Thom polynomials $\Tp_k^{0}$ for $k \le 8$.} 
\label{knownthom}
\small
\begin{tabular}{|l|}
\hline
$\Tp_1^0=c_1$ \\
\hline
$\Tp_2^0=c_1^2+c_2$ \\
\hline
$\Tp_3^0=c_1^3+3c_1c_2+2c_3$\\
\hline
$\Tp_4^0=c_1^4+6c_1^2c_2+2c_2^2+9c_1c_3+6c_4$\\
\hline
$\Tp_5^0=c_1^5+10c_1^3c_2+25c_1^2c_3+10c_1c_2^2+38c_1c_4+12c_2c_3+24c_5$\\
\hline
$\Tp_6^0=c_1^6+15c_1^4c_2+55c_1^3c_3+30c_1^2c_2^2+141c_1^2c_4+79c_1c_2c_3+
5c_2^3+202c_1c_5+55c_2c_4$\\
$\ \ \ \ \ \ \ \ \ +17c_3^2+120c_6$\\
\hline
$\Tp_7^0=c_1^7+21c_1^5c_2+105c_1^4c_3+70c_1^3c_2^2+399c_1^3c_4+
301c_1^2c_2c_3+35c_1c_2^3+960c_1^2c_5+$\\
$\ \ \ \ \ \ \ \ \ +467c_1c_2c_4+139c_1c_3^2+58c_2^2c_3+1284c_1c_6+326c_2c_5+
154c_3c_4+720c_7$\\
\hline
$\Tp_8^0=c_1^8+28c_1^6c_2+140c_1^4c_2^2+140c_1^2c_2^3+14c_2^4+182c_1^5c_3+868c_1^3c_2c_3+501c_1c_2^2c_3+$\\
$\ \ \ \ \ \ \ \ \ +642c_1^2c_3^2+202c_2c_3^2+952c_1^4c_4+2229c_1^2c_2c_4+
364c_2^2c_4+1559c_1c_3c_4+332c_4^2+$\\
$\ \ \ \ \ \ \ \ \ +3383c_1^3c_5+3455c_1c_2c_5+954c_3c_5+7552c_1^2c_6+
2314c_2c_6+9468c_1c_7+5040c_8$\\
\hline
\end{tabular}\\
\end{table}
All coefficients are positive in the table, suggesting Conjecture \ref{conj} (1). From the residue formula \eqref{thompolynomial} we get that for $m=n$ and $1\le i_1 \le i_2 \le \ldots \le i_k \le k, i_1+\ldots +i_k=k$
\begin{equation*}
\coeff_{c_{i_1}c_{i_2}\ldots c_{i_k}}\Tp_k^0=\sum_{\s \in \mathcal{S}_{s \to k}} \Tp_{i_1e_{\s(1)}+\ldots +i_ke_{\s(k)}-\mathbf{1}}
\end{equation*}
where $\mathcal{S}_{s \to k}$ is the set of injective maps $\{i_1,\ldots, i_k\} \to \{1,\ldots, k\}$. 

In particular,
$\coeff_{c_1^k}\Tp_k^0=\Tp_{\mathbf{0}}$, and if the positivity conjecture Conjecture \ref{conj} (1) holds then for $\bj=\bi-\mathbf{1}$ we have $\Sigma \bj=0$ and 
\[\frac{\Tp_{\bj}}{\Tp_\mathbf{0}}<\frac{\coeff_{c_{j_1+1}c_{j_2+1}\ldots c_{j_k+1}}\Tp_k^0}
{\coeff_{c_1^k}\Tp_k^0}.\]
If $\bj=\bj^+-\bj^-$ is the difference of positive vectors $\bj^+,\bj^- \in \ZZ_{\ge 0}^n$ then the right hand side is less then $k^{2\Sigma \bj^+}$ in the listed cases supporting the inequality in Conjecture \ref{conj} (2).
\subsection{Checking $k=3$ for any codimension $m-n$.}\label{subsec:check3} 
Since $Q_3(z_1,z_2,z_3)=1$, the Thom series for $k=3$ is given as
\begin{equation}\nonumber
\Tp_3(z_1,z_2,z_3)=\frac{(z_1-z_2)
(z_1-z_3)(z_2-z_3)}{(2z_1-z_2)(z_1+z_2-z_3)(2z_1-z_3)}
\end{equation}
We observe that the expansion of a fraction of the
form $(1-f)/(1-(f+g))$ with $f$ and $g$ small has positive
coefficients. Indeed, this follows from the identity
\[   \frac{1-f}{1-f-g}=1+\frac{g}{1-f-g}.
\]
Now, introducing the variables $a=z_1/z_2$ and $b=z_2/z_3$, we can
rewrite $\Tp_3$ as follows:
\[\frac{(z_1-z_2)
(z_1-z_3)(z_2-z_3)}{(2z_1-z_2)(z_1+z_2-z_3)(2z_1-z_3)}
=\frac{1-a}{1-2a}\cdot\frac{1-ab}{1-2ab}\cdot\frac{1-b}{1-b-ab}.
\]
Applying the above identity to the right hand side of this formula
immediately implies the first part of Conjecture \ref{conj} for $k=3$.

We leave as an exercise to the reader to show that in this case $\frac{\Tp_{\bi+e_l-e_m}}{\Tp_{\bi}}<9$ holds for any $\bi$ with $\Sigma \bi=0$ and $1\le l<m\le 3$ and that the connectivity of positive coefficients holds.

\end{document}